\documentclass[review,onefignum,onetabnum]{siamart190516}

\usepackage{subfig}
\usepackage{amsmath}
\usepackage{amssymb}
\usepackage{lipsum}
\usepackage{amsfonts}
\usepackage{graphicx}
\usepackage{epstopdf}
\usepackage{algorithmic}
\ifpdf
  \DeclareGraphicsExtensions{.eps,.pdf,.png,.jpg}
\else
  \DeclareGraphicsExtensions{.eps}
\fi

\usepackage{enumitem}

\usepackage{cite}

\nolinenumbers

\newsiamremark{remark}{Remark}
\newsiamremark{hypothesis}{Hypothesis}
\crefname{hypothesis}{Hypothesis}{Hypotheses}
\newsiamthm{claim}{Claim}
\newsiamthm{assumption}{Assumption}

\theoremstyle{plain}
\theoremheaderfont{\normalfont\sc}
\theorembodyfont{\normalfont\itshape}
\theoremseparator{.}

\newcommand{\vect}[1]{\boldsymbol{#1}}

\headers{Equilibria analysis of a networked bivirus model}{B. D. O. Anderson and M. Ye}

\title{Equilibria analysis of a networked bivirus epidemic model using Poincar\'e--Hopf and Manifold Theory\thanks{Submitted to the editors: \today.
\funding{M. Ye was supported in part by the Western Australian Government through the Premier's Science Fellowship Program. }}}

\author{Brian D.O. Anderson\thanks{School Engineering, Australian National University, Acton 2601, A.C.T., Australia
  (\email{brian.anderson@anu.edu.au}).}
\and Mengbin Ye\thanks{Curtin Centre for Optimisation and Decision Science, Curtin University, Perth 6102, WA, Australia 
  (\email{mengbin.ye@curtin.edu.au}).}
}

\usepackage{amsopn}

\ifpdf
\hypersetup{
  pdftitle={To be decided later},
  pdfauthor={M. Ye and B. D. O. Anderson}
}
\fi

\begin{document}

\maketitle
\begin{abstract}
    This paper considers a deterministic Susceptible-Infected-Susceptible (SIS) networked bivirus epidemic model (termed the bivirus model for short), in which two competing viruses spread through a set of populations (nodes) connected by two graphs, which may be different if the two viruses have different transmission pathways. 
    The networked dynamics can give rise to complex equilibria patterns, and most current results identify conditions on the model parameters for convergence to the healthy equilibrium (where both viruses are extinct) or a boundary equilibrium (where one virus is endemic and the other is extinct). However, there are only limited results on coexistence equilibria (where both viruses are endemic). This paper establishes a set of ``counting'' results which provide lower bounds on the number of coexistence equilibria, and perhaps more importantly, establish properties on the local stability/instability properties of these equilibria. In order to do this, we employ the Poincar\'e-Hopf Theorem but with significant modifications to overcome several challenges arising from the bivirus system model, such as the fact that the system dynamics do not evolve on a manifold in the typical sense required to apply Poincar\'e-Hopf Theory. Subsequently, Morse inequalities are used to tighten the counting results, under the reasonable assumption that the bivirus system is a Morse-Smale dynamical system. Numerical examples are provided which demonstrate the presence of multiple attractor equilibria, and multiple coexistence equilibria.
\end{abstract}

% REQUIRED
\begin{keywords}
  susceptible-infected-susceptible (SIS) model, networked systems, dynamical systems, Morse theory, competitive virus spreading
\end{keywords}

% REQUIRED (see https://mathscinet.ams.org/mathscinet/msc/msc2020.html to check)
\begin{AMS}
  34D05, 34C12, 37C65, 92D30, 53C80, 53Z10, 37D15
\end{AMS}

\section{Introduction}

The use of mathematical models to study the spreading of infectious diseases within a population has a long history within the field of epidemiology~\cite{brauer2008mathematical}. One key use of such models is to determine the long term dynamics of the disease of interest as a function of various model parameters, such as the rate of infection and rate of recovery. Deterministic compartmental models have proved especially popular due to their balance between analytic tractability, low cost for numerical simulations, and reasonable accuracy in capturing epidemics. The Susceptible--Infected--Susceptible (SIS) paradigm is a fundamental compartmental paradigm, where each individual in a population is assumed to be healthy and susceptible to infection, or infected and capable of transmitting to susceptible others. Infected individuals can recover, but have no immunity (temporary or permanent) and can be immediately reinfected.  

While SIS models for the spread of a single disease have received significant attention~\cite{lajmanovich1976SISnetwork,shuai2013epidemic_lyapunov}, there has been in the past two decades an increasing interest in models that capture the spread of multiple diseases~\cite{wang2019coevolution}. Cooperative diseases are those in which infection with one disease increases the likelihood of infection with another disease, while competitive diseases are those in which an individual can only be infected with one disease at any one time. The networked bivirus SIS model is one of the more widely studied competitive multivirus models. This model assumes that two competitive diseases, termed virus~$1$ and virus~$2$, spread through a set of nodes, with potentially two different network topologies capturing the different spreading patterns of the two viruses. Virtually all studies have assumed that a form of connectivity holds for the two network topologies. The same bivirus model, i.e., the same set of ordinary differential equations, has been studied in different contexts, with nodes representing gendered groups (two female and one male group) in \cite{carlos2}, individuals in \cite{sahneh2014competitive}, and large populations of constant size in \cite{ye2021_bivirus,santos2015bi}.

There has been a significant amount of literature devoted to the networked bivirus SIS model~\cite{carlos2,sahneh2014competitive,ye2021_bivirus,santos2015bi,prakash2012winner,liu2019bivirus,pare2021multi,janson2020networked,yang2017bi}-- see \cite{ye2022competitive} for a brief survey. In summary, a complete understanding is available for the case of $n=1$~\cite{prakash2012winner} and $n=2$~\cite{ye2021_bivirus}, and for a special $n=3$ case in \cite{carlos2}, but there are still important gaps in understanding the bivirus dynamics for arbitrary $n \geq 3$ node networks. In \cite{ye2021_bivirus}, it was established that for generic model parameters, the number of equilibria are finite, and convergence to an equilibrium occurred for almost all initial conditions. There are nongeneric parameters which instead yield a connected continuum of equilibria. Assuming however genericity of the parameters, there are three \textit{types} of equilibria: i) the ``healthy equilibrium'' (in which both viruses are extinct in every node) and this is always unique, ii) equilibria where virus~$1$ is present and virus~$2$ is extinct (or vice versa), these being termed ``boundary equilibria'', and iii) ``coexistence equilibria'' in which both viruses are present. It turns out that there are at most two boundary equilibria (one for each virus being extinct)~\cite{ye2022competitive}.

Much of the literature has focused on elimination of both viruses and ``survival-of-the-fittest'' outcomes, in which one virus goes extinct while the other persists. Necessary and sufficient conditions for global stability of the healthy equilibrium~\cite{liu2019bivirus,pare2021multi} and local exponential stability of boundary equilibria~\cite{ye2021_bivirus,sahneh2014competitive} are well documented, as are some sufficient conditions for global stability of boundary equilibria~\cite{santos2015bi,ye2021_bivirus}. On the other hand, coexistence equilibria are understudied except for highly nongeneric parameter values~\cite{liu2019bivirus,pare2021multi,ye2021_bivirus}. In fact, while some sufficient conditions for there to be \textit{no} coexistence equilibria are known~\cite{ye2021_bivirus,santos2015bi,liu2019bivirus}, there are only a few sufficient conditions for the existence of coexistence equilibria~\cite{janson2020networked}. More importantly, there are no general analytical results that allow one to relate a given set of generic model parameter values to the number of coexistence equilibria and their local stability properties. This is in part because, with arbitrary $n$ nodes, the dynamics are captured by $2n$ coupled nonlinear differential equations (with the network topologies adding further complexities to the analysis). As a demonstration of the potential complexity of the equilibria patterns, we will in the sequel present two simple $n=4$ examples which have four attractive equilibria (two boundary equilibria and two coexistence equilibria) and two attractive equilibria (one boundary and one coexistence), respectively. 

The limitations discussed motivate us to develop new ``counting'' results for studying coexistence equilibria and their stability properties. We do this using a tool, the Poincar\'e-Hopf Theorem, see e.g. \cite{guillemin2010differential,milnor1997topology}, which enables us to derive one of our main results: a counting result involving the number of equilibria of different indices. Namely, we obtain a constraint on the number and types of coexistence equilibria, (e.g. stable attractor, source, saddle point of certain index). We remark that the Poincar\'e-Hopf Theorem allows an elegant derivation of the principal results applying to the single virus networked SIS model, see \cite{ye2021applications,anderson2020extrema}. However, as will later be clear, the region of interest relevant to a networked bivirus system cannot be so straightforwardly treated as in the single virus case, due to the fact that this region is not a manifold in the sense to which the Poincar\'e-Hopf Theorem applies. In other words, as far as equilibria of the underlying equations are concerned, bivirus is not a simple extension of single virus, and we employ two key ideas to overcome this challenge. The first is that the region of interest can be distorted if necessary at its boundary, to incorporate any stable equilibria located on the boundary, and the second is that a homeomorphism can be established between the modified region of interest and an even-dimensional sphere excluding a single point of that sphere (a fact which then makes it much more straightforward to use the Poincar\'e-Hopf Theorem). As part of our application of the Poincar\'e-Hopf Theorem, we also prove that for generic model parameters, each equilibrium is isolated (and hence there are a finite number), nondegenerate, and hyperbolic. We then make a reasoned argument as to why the networked bivirus SIS model is a so-called Morse-Smale system \cite{smale1967differentiable} (but offer no rigorous proof), before exploiting Morse-Smale inequalities to further strengthen our ability to count the number of coexistence equilibria and determine their stability properties.

The paper is structured as follows. In \Cref{sec:background}, after notational, linear algebra and graph theoretical preliminaries we review the single virus and bivirus equations, and present a motivating example with multiple attractive coexistence equilibria. 
We also introduce the Poincar\'e-Hopf Theorem and indicate the difficulties in immediately applying it. \Cref{sec:main} establishes the first main result of the paper using the Poincar\'e-Hopf Theorem. The use of Morse-Smale inequalities to develop further counting results occurs in \Cref{sec:morse_smale}. A brief illustration of the results is presented in \Cref{sec:examples}, by considering several bivirus systems with two nodes. Finally, conclusions are drawn in \Cref{sec:conclusions}.

\section{Background Material, Epidemic Dynamics and Poincar\'e-Hopf Theory}\label{sec:background}

\subsection{Notation, Linear Algebra and Graph Theory Background}\label{ssec:background}

For real vectors $x,y\in\mathbb R^n$ with entries $x_i$ and $y_i$, $x\geq y$ denotes $x_i\geq y_i$, $x>y$ denotes $x\geq y$ but $x\neq y$, and $x\gg y$ denotes $x_i>y_i$ for all $i=1,2,\dots,n$. For matrices $A,B$ of the same dimensions, $A\geq B, A>B, A\gg B$ denote the same thing as the corresponding inequalities for $\text{vec}\, A$, $\text{vec}\, B$. A matrix $A \in \mathbb R^{n\times n}$ is said to be nonnegative if $\text{vec}\, A \geq  \vect 0_{n^2}$. The $n$-vectors of all 1's and all 0's are denoted by $\vect 1_n$ and $\vect 0_n$ respectively.

For a square matrix $A$, $\rho(A)$ denotes the spectral radius and $\sigma(A)$ the spectral abscissa, or greatest real part of any eigenvalue of $A$. The matrix is termed reducible (irreducible) if there exists (does not exist) a permutation matrix $P$ such that $P^{\top}AP$ is block triangular. For a nonnegative and nonzero $A$, $\sigma(A)=\rho(A)$ and corresponds to a real eigenvalue for which there exist associated left and right eigenvectors which can be taken to be nonnegative; if in addition, $A$ is irreducible, $\rho(A)$ is a unique eigenvalue and these eigenvectors can be taken to be positive and unique up to a scaling. A matrix $-D+A$ with $A$ nonnegative and irreducible and $D$ diagonal has $\sigma(-D+A)$ as a simple eigenvalue, and the associated left and right eigenvectors can be taken to be positive. These facts come from the Perron-Frobenius Theorem and its corollaries, see e.g \cite{horn1994topics_matrix}. Additionally, when $D$ is also positive definite,  $\sigma(-D+A)>0\iff\rho(D^{-1}A)>1, \sigma(-D+A)=0 \iff \rho(D^{-1}A)=1$ and $\sigma(-D+A)<0\iff \rho(D^{-1}A)<1$~\cite[Proposition~1]{liu2019bivirus}. Irreducible nonnegative $A$ have the following property: if $Ax=y$ for $x>{\bf{0}}_n$, $y>{\bf{0}}_n$, then $y$ cannot have a zero entry in every position where $x$ has a zero entry (but it may have zero entries in some of those positions).

A weighted directed graph $\mathcal G$ is a triple $\mathcal G=(\mathcal V,\mathcal E, W)$ with $\mathcal V=\{1,2,\dots,n\}$ denoting the vertex (or node) set, $\mathcal E\subset\mathcal V\times \mathcal V$ denoting the edge set, and $A$ a nonnegative $n\times n$ matrix with $a_{ij}>0$ if and only if $(j,i)\in \mathcal E$, implying the existence of a directed edge from vertex $j$ to vertex $i$.  A path is a sequence of edges of the form $(i_0,i_1),(i_1,i_2),\dots,(i_{p-1},i_{p})$ with the vertices distinct, except possibly for the first and last. A (directed) graph is strongly connected if and only if any one vertex can be joined to any other vertex by a path starting from the first and ending at the second. Further,  $\mathcal G$ is strongly connected if and only if $A$ is irreducible~\cite{berman1979nonnegative_matrices}. 

For a set $X$, typically in this paper a subset of $\mathbb R^n$, the closure, interior and boundary are denoted by $\bar X, X^{\circ}$ and $\partial X$ respectively.

\subsection{Dynamics of single virus networks}\label{ssec:single_virus}
The spreading dynamics of a single virus have been studied using deterministic susceptible-infectious-susceptible (SIS) network models for a long time, see e.g. \cite{lajmanovich1976SISnetwork,fall2007SIS_model,vanMeighem2009_virus,mei2017epidemics_review,shuai2013epidemic_lyapunov}. We summarize the results in a manner relevant to the treatment of bivirus problems below. Assume there are $n$ populations, corresponding to vertices or nodes of a directed graph, each of a large and constant size. Individuals in each population can exist in and move between one of two mutually exclusive \textit{health compartments}: individuals may be healthy and susceptible (S) to becoming infected by the virus, or infected (I) with the virus and able to transmit it. The edges of the graph capture virus transmission possibilities; there is an edge from node $j$ to node $i$ precisely when virus transmission can occur from the infected individuals in the $j$-th population to the susceptible individuals in the $i$-th population. The rates of infection are captured by nonnegative $\beta_{ij}$, so that $\beta_{ij} > 0$ if and only if $(j,i)$ is an edge in the graph (and thus the $\beta_{ij}$ can be regarded as weights on the edges of the underlying graph). Individuals infected with the virus can recover (with no immunity and immediate susceptibility to infection again): the recovery (healing) rate of each population $i$ is captured by the positive parameter $\delta_i$. Let $x_i(t)$, the $i$-th entry of a vector $x(t)\in\mathbb R^n$, denote the fraction of individuals of population $i$ infected with the virus, and let $D={\rm{diag}}(\delta_1,\delta_2,\dots,\delta_n), B=(\beta_{ij})$ and $X(t)={\rm{diag}}(x_1(t),x_2(t),\dots,x_n(t))$. The dynamical equations describing the SIS network system then become
\begin{equation}\label{eq:singlevirusunlabelled}
\dot x(t)=[-D+(I-X(t))B]x(t).
\end{equation}
Since the standard assumption is that $\delta_i > 0$ for all $i$, this means that the diagonal $D$ is assumed to be positive definite. Intuitively, this naturally ensures that each population could recover from either virus if infection transmissions were totally prevented. It is also standard to assume $B$ is irreducible, which is equivalent to $\mathcal{G} = (\mathcal{V}, \mathcal{E}, B)$ being strongly connected; this assumption implies that there is a path of transmission for the virus from any population to any other population.

The key properties of \eqref{eq:singlevirusunlabelled} established by the literature~\cite{lajmanovich1976SISnetwork,fall2007SIS_model,mei2017epidemics_review} include a focus on asymptotic behavior and can be summed up as follows:

\begin{theorem}\label{thm:singlevirus}
With notation as given above, consider the single virus equation~\eqref{eq:singlevirusunlabelled}. Suppose that ${\bf{0}}_n\leq x(0)\leq {\bf{1}}_n$, and the graph $\mathcal{G} = (\mathcal{V}, \mathcal{E}, B)$ is strongly connected. Then ${\bf{0}}_n\leq x(t)\leq {\bf{1}}_n,\,\forall\,t\geq 0$. Moreover:
\begin{enumerate}
    \item If $\rho(D^{-1}B)\leq 1$, all trajectories converge asymptotically to the healthy equilibrium $\vect 0_n$ as $t\to\infty$. Convergence is exponentially fast iff \mbox{$\rho(D^{-1}B)< 1$.} 
    \item If $\rho(D^{-1}B)> 1$, then there is precisely one other equilibrium, $\bar x$, besides the healthy equilibrium $\vect 0_n$. All trajectories converge exponentially fast to $\bar x$ as $t\to\infty$ except if $x(0) = \vect 0_n$. The equilibrium $\bar x$ satisfies $\vect 0_n \ll \bar x \ll \vect 1_n$ and is called the endemic equilibrium.
\end{enumerate}
\end{theorem}

The bounds on $x(t)$ reflect its physical interpretation as a vector of proper fractions, and, with one exception, all trajectories go to the same equilibrium irrespective of initial conditions, viz. the endemic equilibrium if it exists, or the healthy equilibrium otherwise. The exception occurs if an endemic equilibrium exists, but the initial condition lies at the healthy equilibrium. 

The literature often defines $\mathcal{R} = \rho(D^{-1}B)$ as the reproduction number of the SIS network. In epidemiology, the reproduction number is the expected number of secondary infections generated in a population of entirely susceptible individuals, by a single infectious individual. Conveniently, $\mathcal{R}$ for \eqref{eq:singlevirusunlabelled} identifies whether the virus will be eliminated or persist as $t\to\infty$, corresponding closely with the epidemiological definition~\cite{anderson1991_virusbook,hethcote2000mathematics}. We remark that \eqref{eq:singlevirusunlabelled} describes a deterministic system, while real-world virus propagation is a stochastic process. In fact, \eqref{eq:singlevirusunlabelled} can be viewed as the mean-field approximation of a stochastic SIS process. However, we do not provide further comment on this issue, which is beyond the scope of our focus, and instead refer the reader to established literature, e.g.~\cite{nowzari2016epidemics,li2012susceptible,van2015accuracy}.

\subsection{Dynamics of bivirus networks}

Moving from the single virus SIS model, we again assume there are $n$ populations but now two competing viruses may be present, called virus~1 and virus~2 for convenience, and there are three mutually exclusive health compartments, see Fig.~\ref{fig:transitions_bivirus}. Individuals may be susceptible, or infected with virus~1, or infected with virus~2, but cannot be infected with both viruses at the same time due to their competing nature. Individuals that recover from infection by either virus  becomes immediately susceptible to infection from either virus again.
Associated with virus~1 and virus~2 are two graphs, $\mathcal{G}^1 = (\mathcal{V},\mathcal{E}^1, B^1)$ and $\mathcal{G}^2 = (\mathcal{V},\mathcal{E}^2, B^2)$, respectively, which share the same node set but need not have the same edge set or edge weights for both viruses, see Fig.~\ref{fig:two_layer}. The nonnegative infection rates $\beta^1_{ij},\beta^2_{ij}$ capture the associated transmission rates, and are weights for the edges in the two graphs, so that $B^1=(\beta^1_{ij})$ and $B^2 = (\beta^2_{ij})$. Each population $i$ has associated with it positive healing rates $\delta ^1_i$ and $\delta^2_i$ for virus~1 and virus~2, respectively. Let $x^1_i(t),x^2_i(t)$ denote the fraction of individuals in population $i$ infected with virus~1 and virus~2 respectively. Because the viruses are competitive, the total fraction of individuals in population $i$ affected with either virus is $x_i^1(t)+x_i^2(t)$ and the fraction of susceptible individuals is $1-x_i^1(t)-x_i^2(t)$. Let $x^1(t),x^2(t)$ denote the associated vectors of fractions of infected individuals through the $n$ populations, and set $D^1={\rm{diag}}(\delta^1_1,\delta ^1_2,\dots,\delta^1_n)$ and $D^2$ similarly; set $X^1(t)={\rm{diag}}(x^1_1(t),x^1_2(t),\dots,x^1_n(t))$ and $X^2(t)$ similarly. Then the dynamical equations for the bivirus network system become
\begin{subequations}\label{eq:bivirus}
\begin{equation}\label{eq:bivirusa}
\dot x^1(t)=[-D^1+\big(I_n-X^1(t)-X^2(t)\big)B^1]x^1(t)
\end{equation}
\begin{equation}\label{eq:bivirusb}
\dot x^2(t)=[-D^2+\big(I_n-X^1(t)-X^2(t)\big)B^2]x^2(t)
\end{equation}
\end{subequations}

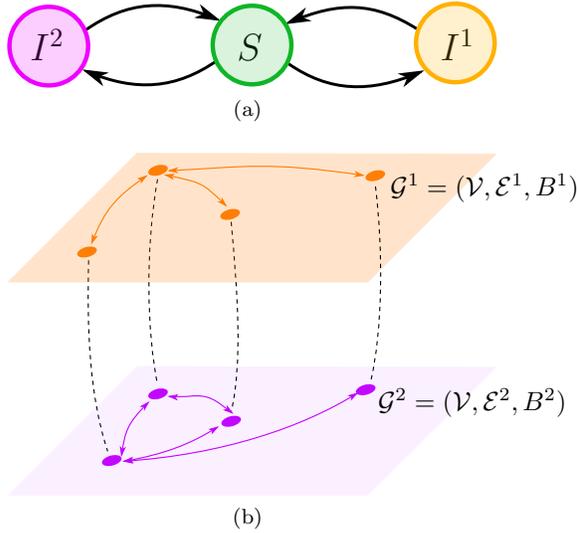
\begin{figure}
\centering
\subfloat[]{\def\svgwidth{0.5\linewidth}
	%% Creator: Inkscape inkscape 0.92.5, www.inkscape.org
%% Accompanies image file 'transitions_bivirus.pdf' (pdf, eps, ps)
\begingroup%
  \makeatletter%
  \providecommand\color[2][]{%
    \errmessage{(Inkscape) Color is used for the text in Inkscape, but the package 'color.sty' is not loaded}%
    \renewcommand\color[2][]{}%
  }%
  \providecommand\transparent[1]{%
    \errmessage{(Inkscape) Transparency is used (non-zero) for the text in Inkscape, but the package 'transparent.sty' is not loaded}%
    \renewcommand\transparent[1]{}%
  }%
  \providecommand\rotatebox[2]{#2}%
  \newcommand*\fsize{\dimexpr\f@size pt\relax}%
  \newcommand*\lineheight[1]{\fontsize{\fsize}{#1\fsize}\selectfont}%
  \ifx\svgwidth\undefined%
    \setlength{\unitlength}{236.77791415bp}%
    \ifx\svgscale\undefined%
      \relax%
    \else%
      \setlength{\unitlength}{\unitlength * \real{\svgscale}}%
    \fi%
  \else%
    \setlength{\unitlength}{\svgwidth}%
  \fi%
  \global\let\svgwidth\undefined%
  \global\let\svgscale\undefined%
  \makeatother%
  \begin{picture}(1,0.17715297)%
    \lineheight{1}%
    \setlength\tabcolsep{0pt}%
    \put(0.46967048,0.05791147){\color[rgb]{0,0,0}\makebox(0,0)[lt]{\lineheight{1.25}\smash{\begin{tabular}[t]{l}\Large $S$\end{tabular}}}}%
    \put(0.04684357,0.05791147){\color[rgb]{0,0,0}\makebox(0,0)[lt]{\lineheight{1.25}\smash{\begin{tabular}[t]{l}\Large $I^2$\end{tabular}}}}%
    \put(0.88493138,0.05791146){\color[rgb]{0,0,0}\makebox(0,0)[lt]{\lineheight{1.25}\smash{\begin{tabular}[t]{l}\Large $I^1$\end{tabular}}}}%
    \put(1.75201837,-0.33917583){\color[rgb]{0,0,0}\makebox(0,0)[lt]{\begin{minipage}{1.13731594\unitlength}\raggedright \end{minipage}}}%
    \put(0,0){\includegraphics[width=\unitlength,page=1]{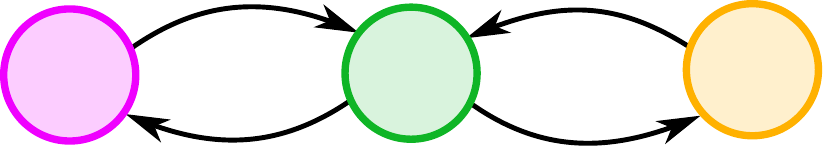}}%
  \end{picture}%
\endgroup%
\label{fig:transitions_bivirus}}
	\hfill
\subfloat[]{\def\svgwidth{0.5\linewidth}
	%% Creator: Inkscape inkscape 0.92.5, www.inkscape.org
%% Accompanies image file 'two_layer_bivirus.pdf' (pdf, eps, ps)
\begingroup%
  \makeatletter%
  \providecommand\color[2][]{%
    \errmessage{(Inkscape) Color is used for the text in Inkscape, but the package 'color.sty' is not loaded}%
    \renewcommand\color[2][]{}%
  }%
  \providecommand\transparent[1]{%
    \errmessage{(Inkscape) Transparency is used (non-zero) for the text in Inkscape, but the package 'transparent.sty' is not loaded}%
    \renewcommand\transparent[1]{}%
  }%
  \providecommand\rotatebox[2]{#2}%
  \newcommand*\fsize{\dimexpr\f@size pt\relax}%
  \newcommand*\lineheight[1]{\fontsize{\fsize}{#1\fsize}\selectfont}%
  \ifx\svgwidth\undefined%
    \setlength{\unitlength}{341.3440063bp}%
    \ifx\svgscale\undefined%
      \relax%
    \else%
      \setlength{\unitlength}{\unitlength * \real{\svgscale}}%
    \fi%
  \else%
    \setlength{\unitlength}{\svgwidth}%
  \fi%
  \global\let\svgwidth\undefined%
  \global\let\svgscale\undefined%
  \makeatother%
  \begin{picture}(1,0.70048434)%
    \lineheight{1}%
    \setlength\tabcolsep{0pt}%
    \put(0,0){\includegraphics[width=\unitlength,page=1]{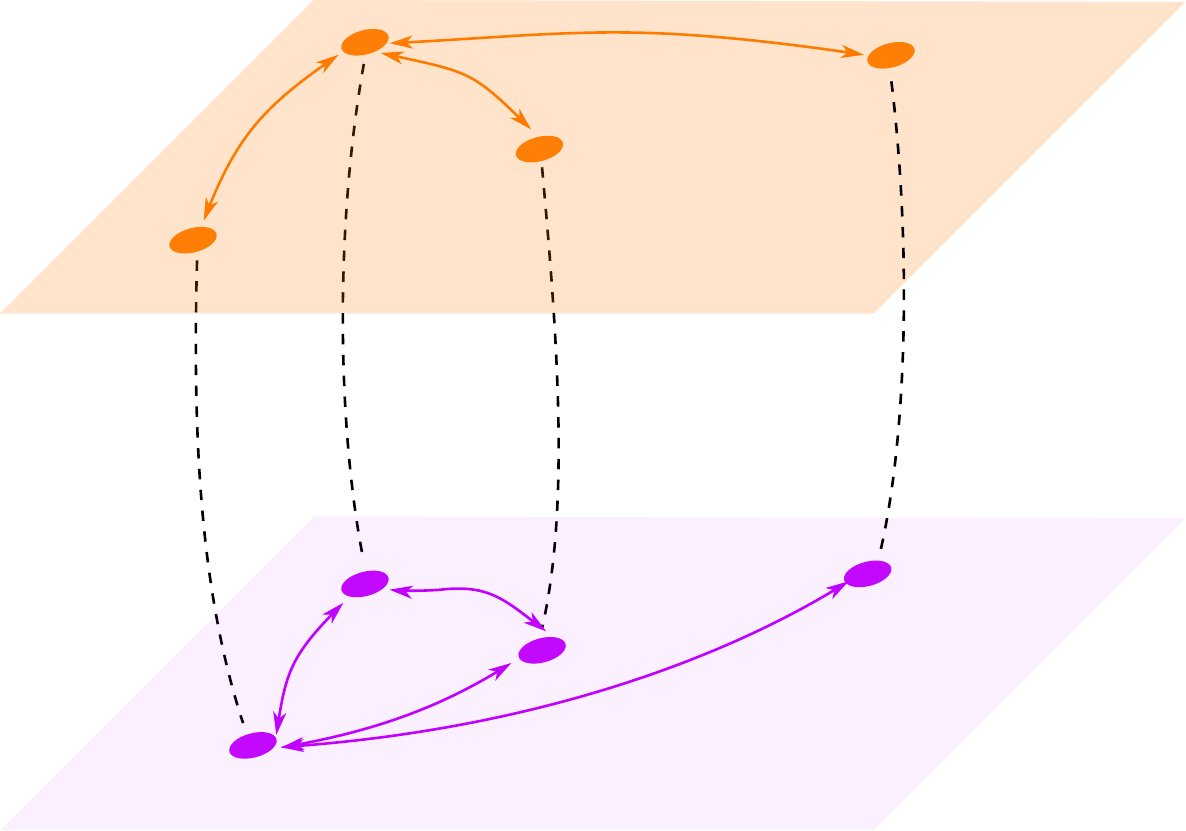}}%
    \put(0.78295643,0.61403911){\color[rgb]{0,0,0}\makebox(0,0)[lt]{\lineheight{1.25}\smash{\begin{tabular}[t]{l}$\mathcal G^1 = (\mathcal{V}, \mathcal{E}^1, B^1)$\end{tabular}}}}%
    \put(0.75765287,0.17483179){\color[rgb]{0,0,0}\makebox(0,0)[lt]{\lineheight{1.25}\smash{\begin{tabular}[t]{l}$\mathcal G^2 = (\mathcal{V}, \mathcal{E}^2, B^2)$\end{tabular}}}}%
  \end{picture}%
\endgroup%
\label{fig:two_layer}}
    \caption{Schematic of the compartment transitions and infection network. (a) Each individual exists in one of three health compartments: Susceptible ($S$), Infected with virus~$1$ ($I^1$, orange), or Infected with virus $2$, ($I^2$, purple). Arrows represent possible transition paths between health states. (b) The network through which the viruses can spread between populations (nodes) is captured by two graphs, $\mathcal{G}^1$ and $\mathcal{G}^2$. Note that the edge sets of the two graphs do not need to match, although the node sets are the same.    }
	\label{fig:epidemic_schematic}
\end{figure}

On physical grounds, the fractions of infected individuals should never move outside the interval $[0,1]$. In fact, that this property is captured by the model is one of the early results in \cite{liu2019bivirus}:

\begin{lemma}
With the above notation and sign constraints on the entries of $D^i$ and $B^i$, suppose that the initial conditions for \eqref{eq:bivirus} satisfy ${\bf{0}}_n\leq x^i(0)\leq {\bf{1}}_n$ for $i=1,2$ and $x^1(0)+x^2(0)\leq {\bf{1}}_n$. Then for all $t>0$, there holds ${\bf{0}}_n\leq x^i(t)\leq {\bf{1}}_n$ for $i=1,2$ and $x^1(t)+x^2(t)\leq {\bf{1}}_n$. 
\end{lemma}

In the sequel, the term `region of interest' will be used to denote the set $$\Xi_{2n}=\{(x^1,x^2) : {\bf{0}}_n\leq x^i\leq {\bf{1}}_n, i=1,2 \text{ and } x^1+x^2\leq {\bf{1}}_n \}.$$

The prime interest in this paper is in the limiting behavior of the equations, and particularly the nature of the equilibria in the region of interest. The situation is certainly more complicated than in the single virus case. To make progress, we now state a standing assumption on the bivirus network, which summarizes the content above concerning assumptions on $D^i, B^i$, and embraces effectively the same assumptions imposed on $D, B$ below \eqref{eq:singlevirusunlabelled} in the single virus case.

\begin{assumption}\label{ass:constraints}
The matrices $D^i$ are positive definite and the matrices $B^i$ are irreducible.
\end{assumption}

Just as in the single virus case, such assumptions are standard in the bivirus network literature, see e.g.~\cite{liu2019bivirus,ye2021_bivirus,santos2015bi,sahneh2014competitive}. We remark that $B^1$ and $B^2$ are assumed separately irreducible, i.e. $\mathcal{G}^1$ and $\mathcal{G}^2$ are both strongly connected but may not share the same edge set or edge weights. 

We now introduce further standing assumptions on the matrices $(D^i)^{-1}B^i$ in order for the bivirus dynamics to be meaningful and interesting. 

\begin{assumption}\label{ass:unstablehealthy}
For $i=1,2$, there holds $\rho((D^i)^{-1}B^i)>1$
\end{assumption}

We briefly describe the reason for these assumptions, but refer the reader to \cite[Section~2.2]{ye2021_bivirus} or \cite[Theorems~1 and 2]{liu2019bivirus} for a detailed treatment and formal proofs. Briefly, if the above assumption fails, we are effectively back in a single virus situation, which has been well explored in the literature, and thus does not deserve further treatment here. For example, suppose the inequality fails for $i = 1$, i.e., $\rho((D^1)^{-1}B^1)\leq 1$. In this case, virus~1 becomes asymptotically extinct\footnote{Exponentially fast in fact if the inequality is strict.}, {\textit{irrespective of the presence of virus 2}}, i.e. $x^1(t)\rightarrow 0$ as $t\rightarrow\infty$. Indeed, presence of virus~2 simply results in a faster extinction for virus~1.

Once $x^1(t)$ has converged to $\vect 0_n$, unsurprisingly given the form of \eqref{eq:bivirusb}, the system behaves like the single virus system with only virus~2 present.
\begin{equation}\label{eq:univirus}
\dot x^2(t)=[-D^2+\big(I-X^2(t)\big)B^2]x(t)
\end{equation}

From \Cref{thm:singlevirus}, we either have $\lim_{t\to\infty }x^2(t) = \vect 0_n$ if $\rho((D^2)^{-1}B^2)\leq 1$ or $\lim_{t\to\infty }x^2(t) = \bar x^2$ if $\rho((D^2)^{-1}B^2)>1$, where $\bar x^2$ is the unique endemic equilibrium for \eqref{eq:univirus}.

In summary then, \Cref{ass:unstablehealthy} provides us the opportunity to study genuine bivirus dynamics, where the persistence and extinction of a virus is tied to the overall networked system, and not to the reproduction number defining it at the single virus level.

With \Cref{ass:unstablehealthy} in place, the equilibria of the bivirus system can all be characterized as one of three types, namely `healthy', `boundary' or `coexistence', as follows. 
\begin{enumerate}
    \item The healthy equilibrium is $(x^1 = \vect 0_n, x^2 = \vect 0_n)$, and it is unstable.
    \item There are precisely two equilibria where one virus is extinct and the other is present: $(\bar x^1, \vect 0_n)$ and $(\vect 0_n, \bar x^2)$, where $\bar x^1$ and $\bar x^2$ correspond to the unique endemic equilibria of the single virus system for virus~1 and virus~2, respectively. Because $(\bar x^1, \vect 0_n)$ and $(\vect 0_n, \bar x^2)$ are on the boundary of the region of interest $\Xi_{2n}$, we shall refer to them loosely as `boundary equilibria', thus being consistent with the literature~\cite{carlos2,ye2021_bivirus}. 
    \item Any other equilibrium $(\tilde x^1, \tilde x^2)$ (if it exists) is termed a coexistence equilibrium, as it must satisfy $\vect 0_n \ll \tilde x^i \ll \vect 1_n$ and $\tilde x^1 + \tilde x^2 \ll \vect 1_n$ (see \cite[Lemma~3.1]{ye2021_bivirus} for the details on the inequality conditions on coexistence equilibria).
\end{enumerate}

Our recent work in \cite{ye2021_bivirus} used monotone systems theory~\cite{smith1988monotone_survey} to establish the limiting dynamical behaviour of generic bivirus networks. Namely, from all initial conditions except possibly a set of measure zero, trajectories will converge to an equilibrium point. Limit cycles, if they exist, must be nonattracting and initial conditions with trajectories converging to a limit cycle correspond to the aforementioned set of measure zero. This limiting behavior closely parallels that of the single virus case, noted in \Cref{thm:singlevirus}, and to the best of our knowledge, no bivirus system has been demonstrated to have a limit cycle. Necessary and sufficient conditions for the boundary equilibria to be locally exponentially stable and unstable were presented in~\cite[Theorem~3.10]{ye2021_bivirus}. Sufficient conditions on the $D^i$ and $B^i$ have been identified yielding all the possible \textit{stability configurations} of the boundary equilibria, that is, conditions such that the resulting bivirus system has neither, one or two of the boundary equilibria being stable~\cite{carlos2,ye2021_bivirus,santos2015bi,santos2015bivirus_conference,janson2020networked}.

With convergence assured and the local stability properties of boundary equilibria fully characterised, key open questions, including our efforts in this work, center around describing the nature and stability properties of the coexistence equilibria and their number. This is a highly nontrivial challenge, partly owing to the network dynamics: the stability configuration of the boundary equilibria does not immediately establish the existence and stability properties (if any) of coexistence equilibria, without further assumptions and conditions on $D^i$ and $B^i$. For example~\cite{ye2021_bivirus}, $D^1 = D^2 = I_n$ and $B^1 > B^2$ is a \textit{sufficient} (but not necessary) condition for i) $(\bar x^1, \vect 0_n)$ to be locally stable and $(\vect 0_n, \bar x^2)$ to be unstable, and ii) no coexistence equilibria\footnote{It turns out this also guarantees that $(\bar x^1, \vect 0_n)$ is globally stable in $\Xi_{2n}\setminus \{(\vect 0_n, \vect 0_n),(\vect 0_n. \bar x^2)\}$.}. It is not clear whether the particular stability configuration of one stable and one unstable boundary equilibria implies the non-existence of coexistence equilibria, or whether non-existence is a consequence of the further constraint $B^1 > B^2$. The aim of this paper is to shed light on how the stability configurations of boundary equilibria help to determine the coexistence equilibria and their stability properties.

For future reference, we note that the Jacobian associated with \eqref{eq:bivirus} evaluated at $(\bar x^1,{\bf{0}}_n)$ is given by

\begin{equation}\label{eq:boundaryjacobian}
    J(\bar x^1,{\bf{0}}_n)=\begin{bmatrix}
    -D^1+(I_n-\bar X^1)B^1-{\rm{diag}}(B^1\bar x^1)&-{\rm{diag}}(B^1\bar x^1)\\
    0&-D^2+(I_n-\bar X^1)B^2
    \end{bmatrix}
\end{equation}
The 1-1 block of the Jacobian is stable, being the same as the Jacobian applying to the steady state (equilibrium) solution of the single virus problem \eqref{eq:singlevirusunlabelled} associated with the parameters $D^1,B^1$, while the 2-2 block may or may not be stable. The same is true \textit{mutatis mutandis} for $J({\bf{0}_n},\bar x^2)$.  One can even construct a special example where both boundary equilibria have 0 as the eigenvalue with largest real part~\cite{liu2019bivirus,ye2021_bivirus}. 

\subsection{A motivating example}\label{ssec:n4_example}

We now present an $n=4$ example bivirus system to motivate the need for tools to derive additional insight into the equilibria.

We set $D^1 = D^2 = I$, and
\begin{equation}
    B^1=\begin{bmatrix} 1.6&1& 0.001 & 0.001 \\ 1&1.6 & 0.001& 0.001 \\ 0.001 & 0.001 & 2.1&0.156\\0.001& 0.001 & 3.0659&1.1
    \end{bmatrix},
    \quad
    B^2=\begin{bmatrix} 2.1&0.156& 0.001 & 0.001 \\ 3.0659&1.1 & 0.001& 0.001 \\ 0.001 & 0.001 & 1&1.6\\0.001& 0.001 & 1.6&1
    \end{bmatrix}.
\end{equation}
As it turns out, there are four locally stable equilibria, of which two are boundary equilibria and two are co-existence equilibria. These equilibria can easily be revealed by, e.g., simulating a number of different initial conditions as recording their limiting points. With $(\tilde x^1, \tilde x^2)$ being an equilibrium, four equilibria are given up to four decimal points by
\begin{equation*}
    \left(\begin{bmatrix}
    0.6157\\
    0.6157\\
    0.5652\\
    0.7160
    \end{bmatrix},\vect 0_4\right)
    \left(\vect 0_4, \begin{bmatrix}
    0.5652\\
    0.7160\\
    0.6157\\
    0.6157
    \end{bmatrix}\right)
    \left(\begin{bmatrix}
    0.0111\\
    0.0065\\
    0.5540 \\
    0.7076
    \end{bmatrix}, 
     \begin{bmatrix}
    0.5540 \\
    0.7076 \\
    0.0111\\
    0.0065
    \end{bmatrix}\right)
      \left(\begin{bmatrix}
    0.6063\\
    0.6005\\
    0.0077\\
    0.0164\\
    \end{bmatrix}, 
     \begin{bmatrix}
    0.0077\\
    0.0164\\
    0.6063\\
    0.6005
    \end{bmatrix}\right)
\end{equation*}
Sample trajectories are given in \Cref{fig:example}.

\begin{figure*}[!htp]
\begin{minipage}{0.475\linewidth}
\centering
\subfloat[Boundary equilibrium with virus~1 endemic\label{fig:n4_boundary1}]{\includegraphics[width=\columnwidth]{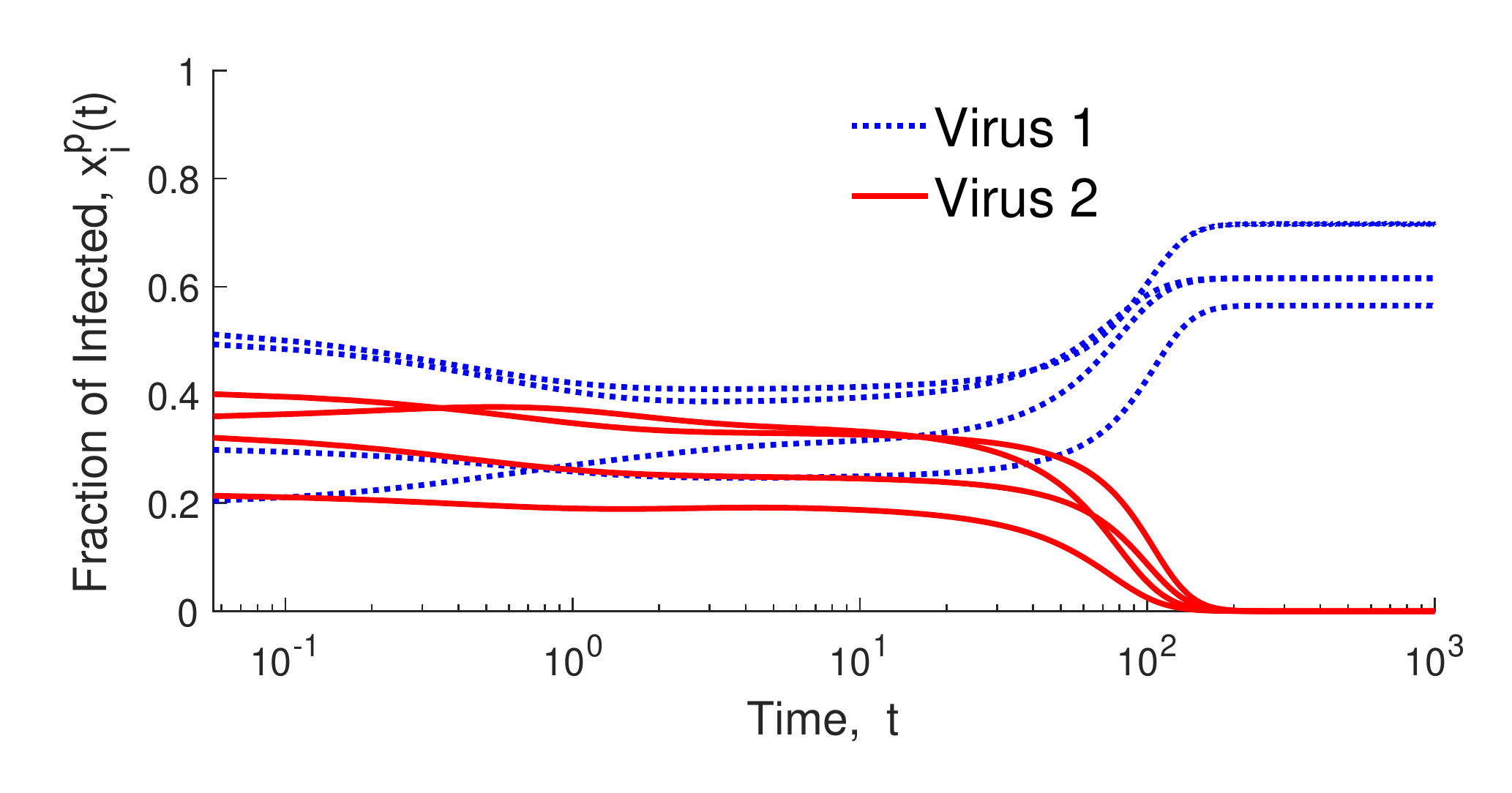}}
\end{minipage}
\hfill
\begin{minipage}{0.475\linewidth}
\centering\subfloat[Boundary equilibrium with virus~2 endemic]{\includegraphics[width=\columnwidth]{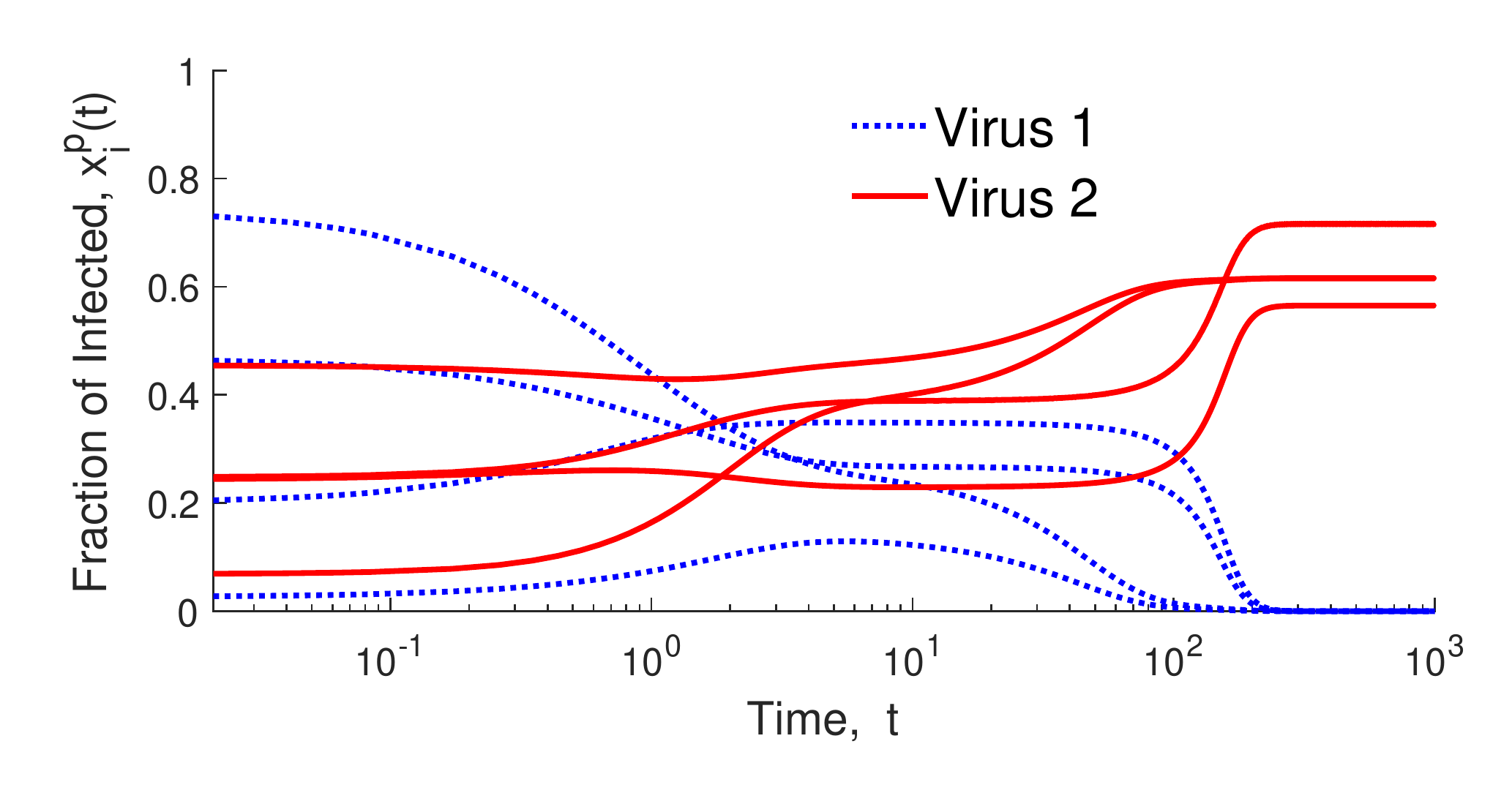}\label{fig:n4_boundary2}}
\end{minipage}
\vfill
\begin{minipage}{0.475\linewidth}
\centering
\subfloat[First attractive coexistence equilibrium]{\includegraphics[width=\columnwidth]{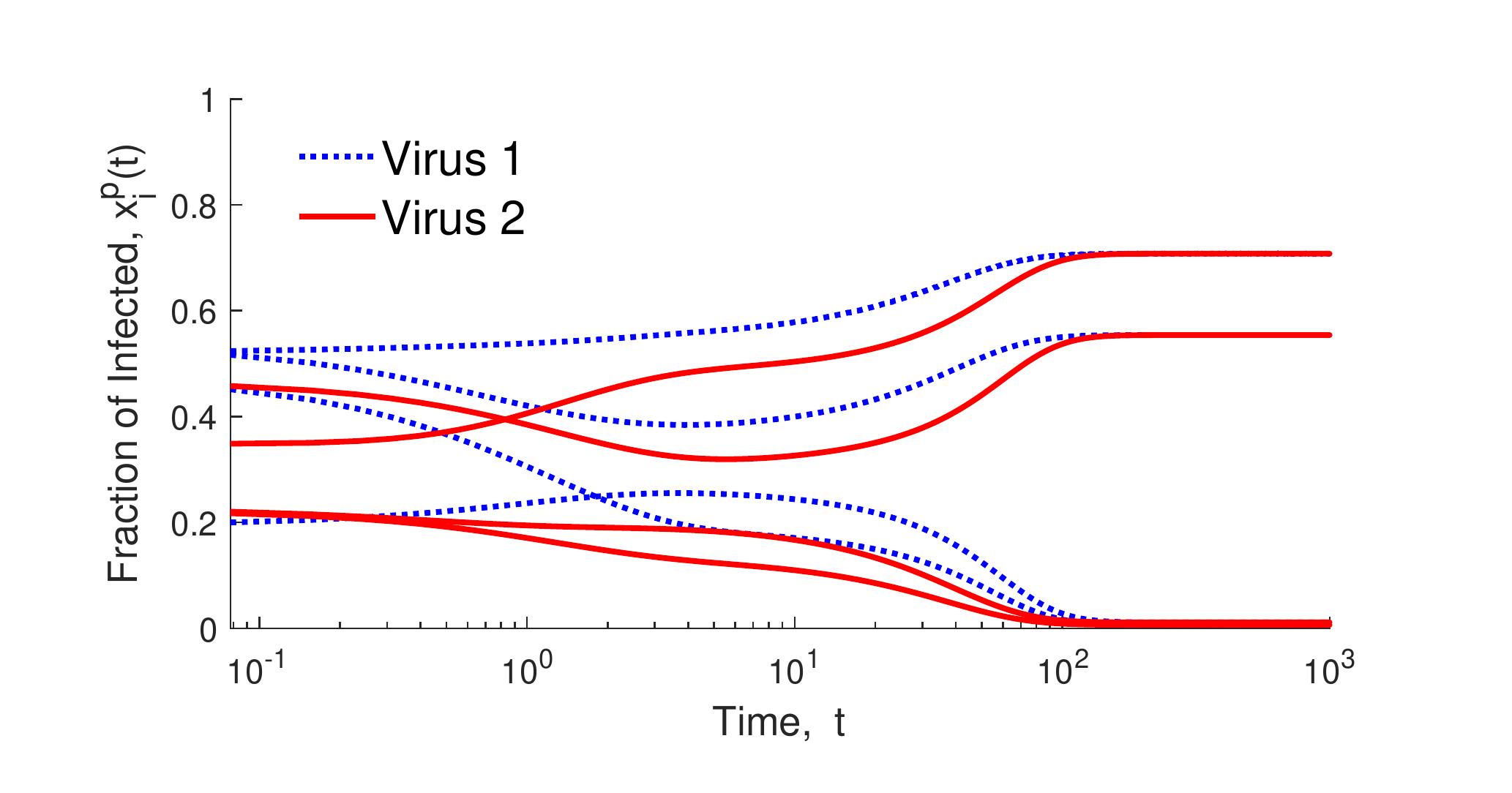}\label{fig:n4_coexistence1}}
\end{minipage}
\hfill
\begin{minipage}{0.475\linewidth}
\centering\subfloat[Second attractive coexistence equilibrium]{\includegraphics[width=\columnwidth]{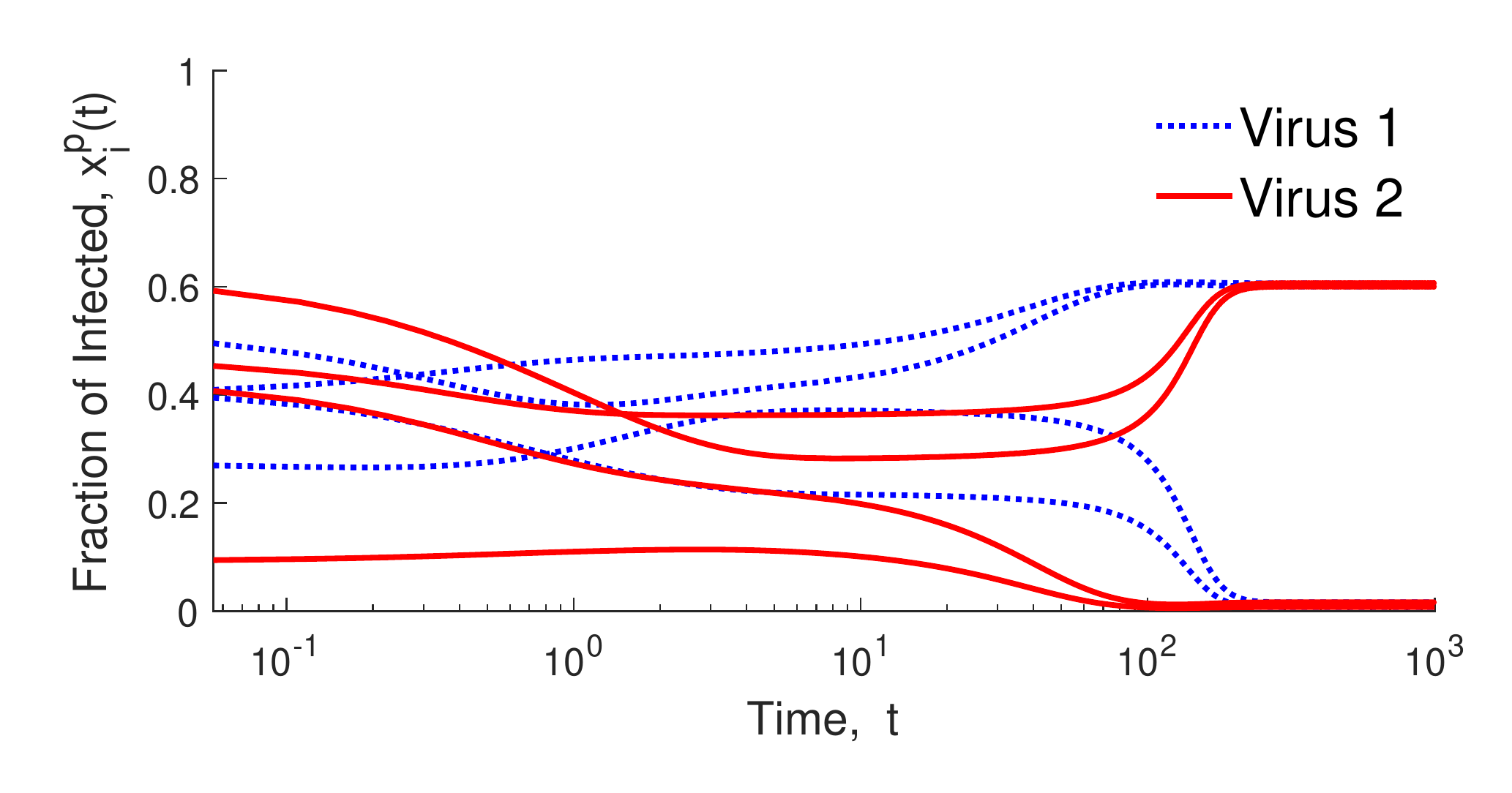}\label{fig:n4_coexistence2}}
\end{minipage}
\caption{Sample trajectories of the $n=4$ example, with different initial conditions. As evident, convergence occurs to four different attractive (locally exponentially stable) equilibria, dependent on the initial conditions. Two are boundary equilibria, where one virus is endemic and the other extinct. Two are coexistence equilibria, where both viruses infect a nonzero fraction of every population. }\label{fig:example}
\end{figure*}

With $n = 4$, analytically solving the equations in \eqref{eq:bivirus} to identify all equilibria in $\Xi_{2n}$ is nontrivial, but this can be achieved with the aid of numerical solvers (or in the case of this particular example, through a combined analytical-numerical approach as explained in \cref{sec:examples} below). The conclusion is that there are in fact five additional unstable coexistence equilibria (reported in \cref{sec:examples}). This numerical example highlights the fact that the equilibria patterns can be complex, even for a system of relatively low dimension, and strongly motivates the development of additional tools to study coexistence equilibria. Numerical solvers are certainly a viable tool for identifying equilibria within $\Xi_{2n}$ when given $(D^i, B^i)$ parameter matrices of modest dimensions as in the example above, but the effectiveness of such solvers is unclear for large-scale networks (where $n$ may be large). Analytical results that allow us to count the number of coexistence equilibria and even possibly determine the number of stable eigenvalues of the Jacobian matrix at the equilibrium, are desirable independent of numerical solver applicability. Such results can offer more general insights into the patterns of coexistence equilibria permissible for generic bivirus systems as opposed to specific numerical examples, and conceivably could be used to, e.g. develop methods for designing bivirus networks with pre-specified equilibrium configurations (see \cref{sec:examples} and additional preliminary results in our pre-print~\cite{ye2021_bivirus_outcomes}).

\subsection{Poincar\'e-Hopf Theorem}\label{ssec:ph_theory}

The version of the Poincar\'e-Hopf Theorem that we will use, which is drawn from \cite[see p. 35 and Lemma 4, p.37]{milnor1997topology}, is as follows (we explain the terminology in detail immediately below):

\begin{theorem}\label{thm:PH}
Consider a smooth vector field on a compact $m$-dimensional manifold $\mathcal M$, defined by the map $f:\mathcal M\rightarrow T\mathcal M$. If $\mathcal M$ is a manifold-with-boundary, then f must point outward at every point on the boundary, denoted by $\partial \mathcal M$. Suppose that every zero $x_k \in\mathcal M$ of $f$ is nondegenerate.\footnote{The most common statement of the Poincar\'e-Hopf Theorem hypothesizes that all zeros are isolated and makes no assumption of nondegeneracy of the vector field zeros  or, equivalently, nonsingularity of the Jacobian at the zero; nor does it involve the signs of the determinants of $df_{x_k}$. Our statement both imposes a tighter condition on the zeros, since nondegeneracy of a zero implies it is isolated, \cite[see p.139] {guillemin2010differential} and obtains a more precise result. Note further that the facts that $\mathcal M$ is compact and every zero is nondegenerate together ensure the number of zeros is finite, so there is no need for a separate explicit requirement for this property in the theorem hypothesis. It is also possible to provide a formulation where at a boundary, a trajectory points inward, rather than outward, but we elect to use what appears to be the much more common formulation of outward-pointing,} Then
\begin{equation}\label{eq:PHmainequation}
\sum_k{\rm{ind}}_{x_k}(f)=\sum_k {\rm{sign \; det}}(df_{x_k})=\chi(\mathcal M)
\end{equation}
\end{theorem}

This theorem statement of course uses language of topology, see e.g. \cite{lee2013introduction,guillemin2010differential}. We provide minimal decoding remarks here. One can think of $f$, a vector field, as the right side of a differential equation 
\begin{equation}\label{eq:general_system}
    \dot x=f(x).
\end{equation}
A point $x_k$ is said to be a zero of $f$ if $f(x_k) = \vect 0$, and thus a zero $x_k$ of $f$ is equivalent to $x_k$ being an equilibrium of \eqref{eq:general_system}. The symbol $T\mathcal M$ denotes the tangent space of the manifold $\mathcal M$. The term `manifold-with-boundary' does have a technical meaning. More precisely, a point $x\in \mathcal M\setminus \partial\mathcal M$ must have the property that there exists a neighborhood of $x$ in $\mathcal M$ that is diffeomorphic to a neighborhood of a point in $\mathbb R^m$, where $m$ is the dimension of $\mathcal M$. Meanwhile, a point $x\in \partial\mathcal M$ (for a manifold-with-boundary) must have a neighborhood in $\mathcal M$ which is diffeomorphic to a half space in $\mathbb R^m$~\cite{guillemin2010differential, lee2013introduction}. As a consequence, a region such as a square (including its edges and corners) in $\mathbb R^2$ would {\textit{not}} qualify, due to the corners not having the requisite property for their neighborhood. Nor could we work with the interior of a square, since it would fail the compactness requirement of the statement of \Cref{thm:PH}. The notion of pointing outward can be rigorously defined using the notion of a `tangent-cone'~\cite{blanchini1999set_invariance,ye2021applications}, but for our purposes the intuitive interpretation is adequate. The notation $df_{x_k}$ denotes the Jacobian of $f$ evaluated at $x_k$ and a nondegenerate zero is defined as one at which the Jacobian is nonsingular, see \cite[p.~37]{milnor1997topology}, with such a zero being necessarily isolated. (Separately, note that nondegeneracy is a property independent of the choice of coordinate system, see e.g. \cite[see p.~42]{matsumoto2002introduction}, where the independence of the sign of the Jacobian determinant is also demonstrated). The index of a vector field is the sum of the indices of its zeros, and the index of a zero of a vector field is a standard concept, see e.g. \cite[Section 3.5]{guillemin2010differential}. Even in as simple a manifold as  $\mathbb R^2$, any positive or negative integer value for the index at a zero is possible; however, with the nondegeneracy assumption, the only possibilities for an arbitrary manifold are $\pm 1$, see~\cite[p.~37, Lemma 4]{milnor1997topology} and \cite[p.~139]{guillemin2010differential}. The Euler characteristic of $\mathcal{M}$, denoted by $\chi(\mathcal{M})$, is determined by the shape of the manifold, and values are known for common manifolds, e.g. ball, sphere, torus, $\mathbb R^n$, etc. 

For future reference we note that if $n_k$ is the number of eigenvalues of the Jacobian $df_{x_k}$ with negative real part, then
\begin{equation}\label{eq:ph_sum}
    {\rm{sign\; det}}(df_{x_k})=(-1)^{n_k}
\end{equation}

Our aim is to use the above version of the Poincar\'e-Hopf Theorem in a bivirus setting. However, seeking to apply the Theorem by identifying the region of interest $\Xi_{2n}$ with $\mathcal M$ immediately creates some problems, including the following. 
\begin{enumerate}
    
    \item
    The theorem requires zeros of the vector field to be nondegenerate (which implies they are isolated, \cite[[see p. 139]{guillemin2010differential}. This implies that with $f$ denoting the vector field, if the manifold in question is bounded, the number of zeros of $f$ or equilibrium points of \eqref{eq:general_system} must be finite. It is clearly preferable for a property such as finiteness of the number of zeros, or nondegeneracy of each zero, to be demonstrated, rather than just assumed. Awkwardly however, for the bivirus dynamics in \eqref{eq:bivirus}, special networks have been identified with an infinite number of equilibria that form a segment of a line in $\Xi_{2n}$~\cite{liu2019bivirus,ye2021_bivirus}. 
    \item 
    The region of interest $\Xi_{2n}$ is not a manifold without boundary, nor a manifold-with-boundary. It is a compact region, and it does have a boundary in a set theory sense, but it fails to be a manifold-with-boundary, essentially because it has corners.  It can in fact be studied as a manifold-with-corners~\cite{lee2013introduction}, but a version of the Poincar\'e-Hopf Theorem for a manifold-with-corners does not appear to be available in the literature.
     \item
    The theorem requires the vector field defined by $f$ to be outward pointing on the boundary $\partial \mathcal M$.  If one believed that the theorem should apply to a manifold-with-corners such as $\Xi_{2n}$, this requirement would preclude the existence of any zeros of the vector field on the boundary, since the direction is simply not defined at a zero (and in an arbitrary neighborhood many directions occur, a fact which probably precludes a limiting argument). There are however three equilibria on the boundary of $\Xi_{2n}$, viz. the healthy equilibrium, and the two boundary equilibria.
    
\end{enumerate}

Our paper will address all three issues. The first issue is comparatively easy to deal with, as it turns out.  Our strategy for dealing with the second issue depends on two major steps; in the first, the region of interest $\Xi_{2n}$ will be perturbed with an enlargement so that any stable equilibrium on its boundary becomes an interior point of the perturbed region. Second, we will exhibit a diffeomorphic map from the interior of the modified region of interest to an even-dimensional sphere, less a single point; because this is a region to which the Poincar\'e-Hopf Theorem is obviously almost applicable, with some further massaging a result can be obtained for the bivirus system. To deal with the third issue, we will actually apply the Poincar\'e-Hopf Theorem to a modified manifold, viz. the $2n$-dimensional sphere, which has no boundary. There are then no trajectory directions that have to be checked.

\section{Main Result}\label{sec:main}

In this section, we return to the direct study of \eqref{eq:bivirus}. As remarked previously, use of Poincar\'e-Hopf theory presupposes that the number of equilibria, at least in $\Xi_{2n}$, of the differential equations (equivalently the zeros of the vector field $f$) is finite, and indeed that all zeros have a nondegeneracy property.  The first novel contribution of the paper establishes that such nondegeneracy is normally to be expected, using differential topology ideas for the proof. It is a fundamental precursor for the two main `counting' results of the paper. 

\begin{theorem}\label{thm:finiteness}
With notation as previously defined, consider the equation set \eqref{eq:bivirus} in the region $\Xi_{2n}$ and assume that \Cref{ass:constraints} and \Cref{ass:unstablehealthy} both hold. Then with any fixed $B^1,B^2$, and the exclusion of a set of values for the entries of $D^1,D^2$ of measure zero,  the number of equilibrium points (equivalently zeros of the associated vector field) is finite and each zero is nondegenerate. Similarly, with any fixed $D^1,D^2$, and the exclusion of a set of values for the entries of $B^1,B^2$ of measure zero, the same property for the equilibrium points is assured. 
\end{theorem}

This second claim of the theorem  can be established by appealing to ideas of algebraic geometry~\cite[Theorem~3.6]{ye2021_bivirus} and indeed a different proof can be found mixing manifold ideas and algebraic geometry ideas in~\cite{robbin1981algebraic}. We provide in \Cref{app:pf_thm_finiteness} a third proof avoiding algebraic geometry entirely, instead appealing to the Parametric Transversality Theorem in manifold theory, see e.g.~\cite[p.~145]{lee2013introduction} and \cite[p.~68]{guillemin2010differential}, and the proof also covers the first part of the theorem as well. Approaches relying on algebraic geometry require the system equations in \eqref{eq:bivirus} to be polynomial in the state variables, while the Parametric Transversality Theorem approach does not. The latter thus offers an advantage for extending analysis to bivirus network dynamics with non-polynomial terms, for instance if we were to consider introduction of feedback control~\cite{ye2021applications}, or certain small smooth variations to the right side of the differential equation, or even the larger variations provided in a modification to the quadratic terms suggested in \cite{yang2017bi}. The requirement that, when the $D^i$ are fixed,  the $B^i$ are excluded from a set of measure zero is indeed needed: as we mentioned in \Cref{ssec:ph_theory}, sets of specially structured $B^i$ exist for which there is a continuum of equilibria. 

Going along with the immediately preceding theorem is a strengthening of the nondegeneracy condition that we will need. Generically, zeros of the vector field are not just nondegenerate, but also hyperbolic. That is, the Jacobian matrix at a zero is free of eigenvalues with zero real part. 

\begin{theorem}\label{app:pf_thm_hyperbolicity}
Adopt the same hypothesis as Theorem~\ref{thm:finiteness}. With any fixed matrices $B^1,B^2$, and the exclusion of a set of values for the entries of $D^1,D^2$ of measure zero, the number of equilibrium points is finite and the associated vector field zero is hyperbolic. Similarly, with any fixed $D^1,D^2$, and the exclusion of a set of values for the entries of $B^1,B^2$ of measure zero, the same property for the equilibrium points is assured. 
\end{theorem}

The proof of this theorem can be found in Appendix~\ref{sec:app b}. The first main result of the paper can now be stated as follows:

\begin{theorem}\label{thm:main}
With notation as previously defined, consider the equation set \eqref{eq:bivirus} and suppose that Assumptions \ref{ass:constraints} and \ref{ass:unstablehealthy} both hold. Suppose that the equilibria of \eqref{eq:bivirus} in the region of interest $\Xi_{2n}=\{{\bf{0}}_n\leq x^i\leq {\bf{1}}_n,  i=1,2 \}\cap\{ x^1+x^2\leq {\bf{1}}_n\}$ are all hyperbolic and thus finite in number. Excluding the healthy equilibrium $({\bf{0}}_n,{\bf{0}}_n)$ and any unstable boundary equilibrium, let $n_k$ denote the number of open left half plane eigenvalues of the Jacobian associated with the $k$-th equilibrium in  $\Xi_{2n}$. Then there holds 
\begin{equation}\label{eq:count}
\sum_k(-1)^{n_k}=1
\end{equation}
\end{theorem}

We note two other references relevant to this theorem. First, the appendix of a paper by Glass \cite{glass1975combinatorial} suggests in an equilibrium classification problem a modification of the Poincar\'e-Hopf formula applied to that problem obtained through similar exclusion of unstable vector field zeros on the boundary of a region of interest. The argument of that paper is more in outline form than provided in full detail, whereas we provide a rigorous treatment. Second, a paper of Hofbauer \cite{hofbauer1990index}, using methods of real analysis as much as topology, establishes an index theorem for a certain class of dynamical systems, into which the bivirus system can be shown to fit. The methods employed do not appear to lend themselves to straightforward extension to Morse-Smale inequalities. 

The proof of this theorem will be developed through some preliminary results, dealing with the three issues described earlier below \eqref{eq:ph_sum}. First, we focus on the behavior of the vector field on the boundary of $\Xi_{2n}$, and demonstrate that an appropriate perturbation of the region of interest will ensure the vector field `points inward' to $\Xi_{2n}$. It is convenient to look at three different types of boundary points of $\Xi_{2n}$:
\begin{enumerate}
    \item 
    Boundary points where $x^1_i+x^2_i=1$ for one or more $i$;
    \item
    Boundary points where $x^2_i=0$ for at least one but not all $i$ (with the same conclusions applying in respect of $x^1_i=0$ for some but not all $i$);
    \item
    Boundary points where $x^2={\bf{0}}_n$ (with the same conclusions applying to $x^1={\bf{0}}_n$), and the healthy equilibrium $(\bf{0}_n,\bf{0}_n)$. 
\end{enumerate}
Establishing this `inward pointing' property on $\Xi_{2n}$ at the aforementioned boundary points is essential for us to subsequently introduce a sphere as the manifold where we will apply the Poincar\'e-Hopf Theorem to the bivirus system. Once we are on the sphere, and with no boundary, the requirement in the Poincar\'e-Hopf Theorem on the vector field `pointing outwards' becomes irrelevant. 
 
 \subsection{Trajectories from the first two types of boundary point}\label{ssec:traj_01}
 The outcomes for the first two cases are summarised in the following two lemmas. 
 
 \begin{lemma}
 Consider a boundary point of $\Xi_{2n}$ where $x^1_i+x^2_i=1$ for some $i$. Then trajectories of \eqref{eq:bivirus} are inward-pointing.
 \end{lemma}
 
 \begin{proof}
 Suppose $x^1_i+x^2_i=1$. The differential equation for $x^1_i$ then immediately yields $\dot x^1_i=-\delta^1_i x^1_i$ and similarly for $x^2_i$, and the claim is immediate. 
 \end{proof}

 \begin{lemma}
 Suppose that $x^1(0)+x^2(0)\ll {\bf{1}}_n$ and that (with inessential reordering if necessary), there holds $x^2_1(0)=0, x^2_2(0)=0,\dots, x^2_p(0)=0,x^2_{p+1}(0)>0,\dots,x^2_{n}(0)>0$ for some $0<p<n$. Then for some $\bar i\in\{1,2,\dots,p \}$, there holds $\dot{\bar x}^2_i>0$, and at an arbitrarily small time $t_p>0$, trajectories obey $\dot x^2(t_p)\gg {\bf 0}_n$.  
 \end{lemma}
 
 \begin{proof}
 Since the matrix $B^2$ is irreducible, and $I-X^1(0)-X^2(0)$ is nonsingular, the matrix $[I-X^1(0)-X^2(0)]B^2$ is also irreducible.  This means that the vector $[I-X^1(0)-X^2(0)]B^2x^2(0)$ cannot have a zero in every position where $x^2(0)$ has a zero, i.e. for one or more $i\in\{1,2,\dots,p\}$, say $i=\bar i$, there holds $[\big (I-X^1(0)-X^2(0)\big )B^2x^2(0)]_{\bar i}>0$, whence also from \eqref{eq:bivirus}, there holds $\dot x^2_{\bar i}>0$. Note that for all other $i\in\{1,2,\dots,p\}$, there necessarily holds $\dot x^2_{ i}(0)=[[I-X^1(0)-X^2(0)]B^2x^2(0)]_i \geq 0$.

 As a consequence of the first part of the lemma, we see that for a time $t_1$ that is arbitrarily small and positive, fewer than $p$ entries of $x^2(t_1)$ will be zero. And then for a time $t_2$ for which $t_2-t_1$ is arbitrarily small and positive, fewer than $p-1$ entries will be zero. Continuing the argument, there exists an arbitrarily small but positive time $t_p$ for which $\dot x^2(t_p)\gg {\bf 0}_n$, which is equivalent to saying that trajectories are inward pointing.  \hfill$\qed$
 \end{proof}

\subsection{Trajectories from the third type of boundary point}\label{ssec:traj_02}

We now consider boundary points where $x^2={\bf{0}}_n$ (with any conclusions drawn also applying to boundary points where $x^1 = {\bf 0}_n$). As explained below \Cref{ass:unstablehealthy}, for any initial condition for which $x^2(0)={\bf{0}}_n$, there will hold $x^2(t)={\bf{0}}_n$, for all $t$. Thus trajectories are neither inward or outward pointing with respect to $\Xi_n$, but remain on the boundary. Further, on the boundary where $x^2={\bf{0}}_n$, there are precisely two equilibria points, viz. the boundary equilibrium $(\bar x^1, {\bf 0}_n)$ with $\bar x^1\gg {\bf 0}_n$ and the healthy equilibrium $({\bf 0}_n, {\bf 0}_n)$. Below we consider first the case where this is a stable equilibrium (all eigenvalues of the Jacobian have negative real part), and subsequently the case where it is unstable (one or more eigenvalues of the Jacobian have positive real part). Following this, we treat the case of the healthy equilibrium, which is unstable by Assumption \ref{ass:unstablehealthy}.

\subsubsection{Perturbation of the region of interest around a stable boundary equilibrium} \label{sssec:stable_boundary}
With the boundary equilibrium $(\bar x^1, {\bf 0}_n)$ locally exponentially stable (as an equilibrium of \eqref{eq:bivirus}), we shall explain how to make a perturbation of the boundary of $\Xi_{2n}$ in the vicinity of $(\bar x^1, {\bf 0}_n)$, defined by a hemisphere extending outwards from the boundary, and joined smoothly to the boundary by a $C^{\infty}$ bump function \cite[see p.127]{tu2011introduction}.

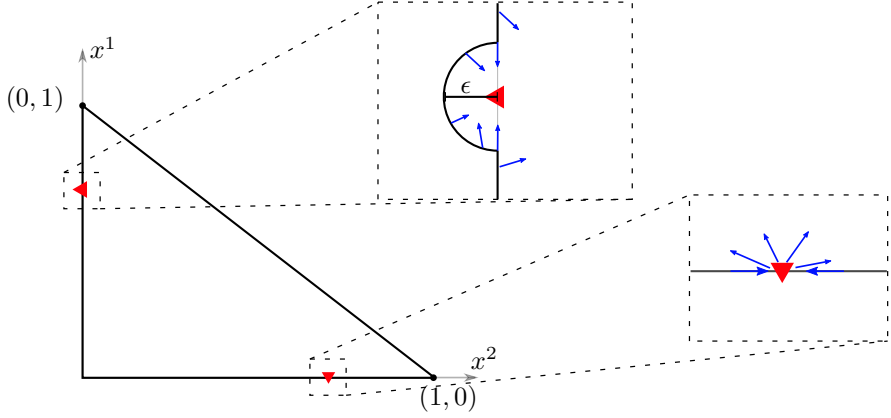
\begin{figure}
    \centering
    \def\svgwidth{0.9\linewidth}
    %% Creator: Inkscape inkscape 0.92.5, www.inkscape.org
%% Accompanies image file 'bivirus_perturbation.pdf' (pdf, eps, ps)
\begingroup%
  \makeatletter%
  \providecommand\color[2][]{%
    \errmessage{(Inkscape) Color is used for the text in Inkscape, but the package 'color.sty' is not loaded}%
    \renewcommand\color[2][]{}%
  }%
  \providecommand\transparent[1]{%
    \errmessage{(Inkscape) Transparency is used (non-zero) for the text in Inkscape, but the package 'transparent.sty' is not loaded}%
    \renewcommand\transparent[1]{}%
  }%
  \providecommand\rotatebox[2]{#2}%
  \newcommand*\fsize{\dimexpr\f@size pt\relax}%
  \newcommand*\lineheight[1]{\fontsize{\fsize}{#1\fsize}\selectfont}%
  \ifx\svgwidth\undefined%
    \setlength{\unitlength}{598.65358193bp}%
    \ifx\svgscale\undefined%
      \relax%
    \else%
      \setlength{\unitlength}{\unitlength * \real{\svgscale}}%
    \fi%
  \else%
    \setlength{\unitlength}{\svgwidth}%
  \fi%
  \global\let\svgwidth\undefined%
  \global\let\svgscale\undefined%
  \makeatother%
  \begin{picture}(1,0.46273708)%
    \lineheight{1}%
    \setlength\tabcolsep{0pt}%
    \put(0.51413119,0.35961556){\color[rgb]{0,0,0}\makebox(0,0)[lt]{\begin{minipage}{0.27114415\unitlength}\raggedright \end{minipage}}}%
    \put(0,0){\includegraphics[width=\unitlength,page=1]{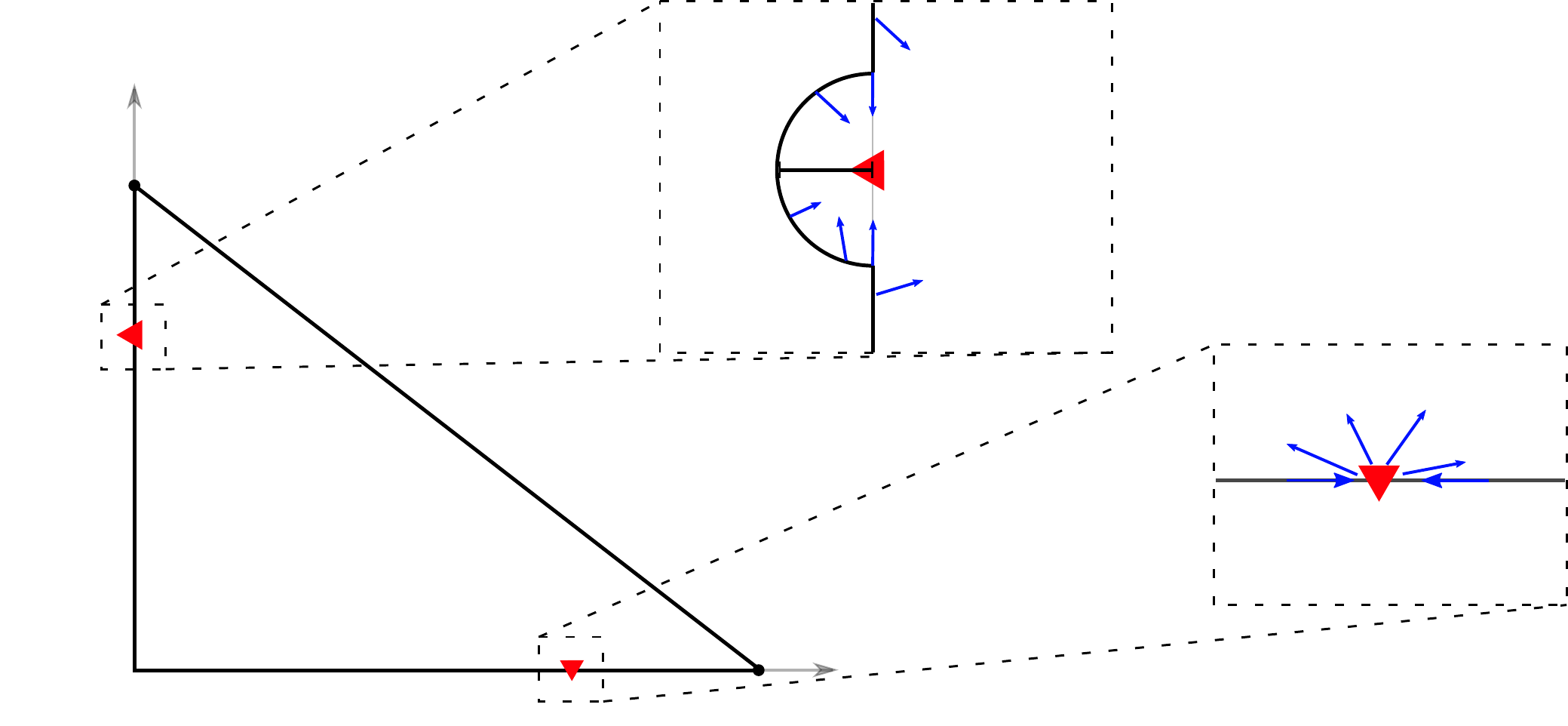}}%
    \put(0.51458432,0.35856971){\makebox(0,0)[lt]{\lineheight{1.25}\smash{\begin{tabular}[t]{l}$\epsilon$\end{tabular}}}}%
    \put(0.09399196,0.39809478){\color[rgb]{0,0,0}\makebox(0,0)[lt]{\lineheight{1.25}\smash{\begin{tabular}[t]{l}$x^1$\end{tabular}}}}%
    \put(-0.00151707,0.34432631){\color[rgb]{0,0,0}\makebox(0,0)[lt]{\lineheight{1.25}\smash{\begin{tabular}[t]{l}$(0,1)$\end{tabular}}}}%
    \put(0.46735857,0.00403248){\color[rgb]{0,0,0}\makebox(0,0)[lt]{\lineheight{1.25}\smash{\begin{tabular}[t]{l}$(1,0)$\end{tabular}}}}%
    \put(0.52643766,0.04541317){\color[rgb]{0,0,0}\makebox(0,0)[lt]{\lineheight{1.25}\smash{\begin{tabular}[t]{l}$x^2$\end{tabular}}}}%
  \end{picture}%
\endgroup%

    \caption{Illustration of perturbation of $\Xi_{2n}$ for $n = 1$, with the triangles indicating boundary equilibria. In this example $(\bar x^1, 0)$ is locally exponentially stable, and $(0, \bar x^2)$ is unstable. Blue arrows indicate the trajectory directions, i.e., the flow of the vector field. At $(\bar x^1, 0)$, a perturbation is introduced, adding  a hemisphere (semicircle for scalar $x^1, x^2$) of radius $\epsilon$ to cover points with $|x^2|<\epsilon$ but smoothly joined to $x^2=0$. By picking $\epsilon$ suitably small, all points in the perturbed region are in the region of attraction of $(\bar x^1, 0)$ and hence on the boundary of the hemisphere, the trajectory points `inward'. This is detailed in \cref{sssec:stable_boundary}.  At $(0, \bar x^2)$, we make no such perturbation. Trajectories beginning with $x^1(0) = 0$ will have $x^1(t) = 0$ for all $t$, i.e. the trajectories move along the $x^2$ axis towards $(0, \bar x^2)$. However, in the interior, all trajectories point away from the unstable equilibrium, as detailed in \cref{sssec:unstable_boundary}.}  \label{fig:bivirus_perturbation}
\end{figure}

To introduce the idea, suppose for the moment there is a single dimension, i.e. $x^2$ is a scalar, and we are working with a bivirus system with just one population. Now the boundary of $\Xi_{2}$ which has $x^2=0$ defines a line, along which $x^1$ varies. At some point on this line lies the equilibrium $(\bar x^1, 0)$, which is locally exponentially stable by hypothesis. Perturb the boundary of $\Xi_2$ along the line $x^2=0$ using a bump function in the vicinity of $(\bar x^1,0)$ so that for  an arbitrary fixed but suitably small positive $\epsilon$, the perturbation occurs within the interval $(\bar x^1-\epsilon,\bar x^1+\epsilon)$, the perturbation is in the direction $x^2<0$, and the perturbation ensures that $|x^2|<\epsilon$ is within the new perturbed boundary.\footnote{In more detail, suppose $f(t)$ is the function $\exp(-1/t), t>0$ and $0$ for $t\leq 0$.  Define $g(t)=f(t)/[f(t)+f(1-t)]$; this is a function which transitions smoothly and monotonically from value 0 at $t=0$ to value 1 at $t=1$. Define $h(t)=g\big((t +\epsilon)(2\epsilon)^{-1}\big)$. This function transitions smoothly and monotonically from 0 at $t=-\epsilon$ to 1 at $t=\epsilon$. The function $\phi_{\epsilon}(t)=4\epsilon h(t)(1-h(t))$ is smooth and transitions with monotone increase from 0 at $t=-\epsilon$ to $\epsilon$ at $t=0$ and transitions with monotone decrease from $\epsilon$ at $t=0$ to zero at $t=\epsilon$. It is zero outside $(-\epsilon, \epsilon)$. 
   }
The boundary $x^2=0$ is replaced  $x^2=-\phi_{\epsilon}(x_1-\bar x_1)$.

The detailed definition of $\phi_{\epsilon}$ is contained in the footnote, and the boundary perturbation is like a semicircle extending into the left half plane, but smoothly connected to the axis $x^2=0$. An illustration of this is presented in \Cref{fig:bivirus_perturbation}.

Here is how to generalize this idea to the case when $x^1,x^2$ are both $n$-dimensional, using the same function $\phi_{\epsilon}$.

\begin{lemma}\label{lem:stableboundary}
Suppose the region $\Xi_{2n}$ is as defined earlier and  with $\epsilon$ an arbitrary positive constant,  denote by $\phi_{\epsilon}$ a smooth bump function that is zero outside of $(-\epsilon,\epsilon)$, positive in $(-\epsilon,\epsilon)$, and taking the value $1$ at $0$. Consider that part of the boundary of $\Xi_{2n}$ defined by $x^2={\bf{0}}_n$, with $(\bar x^1, {\bf{0}}_n)$ being the boundary equilibrium of the bivirus system. Now with $\epsilon$ a suitably small arbitrary but fixed positive constant, expand the region by defining a new boundary via
\begin{equation}\label{eq:perturb}
    x^2_i=-\frac{1}{\sqrt n}\phi_{\epsilon}(\|x^1-\bar x^1\|)\,,\;\forall\, i\in\mathcal{V}
\end{equation}
Then points on the new boundary either have all entries of $x^2$ negative, or all entries zero, and those with $x^2\ll {\bf{0}}_n$  are all no greater than a distance $\epsilon$ from $(\bar x^1,{\bf{0}}_n)$ and obey $\|x^2\| \leq \epsilon$. 
\end{lemma}
  
Now choose $\epsilon$ so that all points in the perturbed region at a distance $\epsilon$ from the equilibrium $(\bar x^1,{\bf{0}}_n)$ are in its region of attraction. Then all points on the modified boundary which have  $x^2\ll {\bf{0}}_n$ lie within the region of attraction of $(\bar x^1, {\bf{0}}_n)$. All other points on the boundary (i.e. those which are also on the boundary of the unmodified region) have an inward pointing trajectory (in the case where one or more entries of $x^2$ are positive), or a trajectory pointing along the boundary towards $(\bar x^1,{\bf{0}}_n)$ (in the case where $x^2={\bf{0}}_n$), by the arguments given previously.
  
Henceforth, we shall use the notation $\tilde{\Xi}_{2n}$ to denote the perturbation of $\Xi_{2n}$ to encompass any locally exponentially stable boundary equilibrium, achieved using \eqref{eq:perturb} in the case of $(\bar x^1, {\bf{0}}_n)$ and/or a similar perturbation in the case of $({\bf 0}_n, \bar x^2)$. 
  
  \subsubsection{Behavior in the vicinity of an unstable boundary equilibrium}\label{sssec:unstable_boundary}
  
  We now examine trajectories in the vicinity of a boundary equilibrium that is unstable. Our exposition will consider $(\bar x^1, {\bf{0}}_n)$, but the same conclusions can be drawn if $({\bf 0}_n, \bar x^2)$ is unstable. No perturbation of the region of interest is made. Unless all eigenvalues of the associated Jacobian matrix have positive real parts, there are always some trajectories which can approach the unstable equilibrium (with such trajectories defining the `stable manifold' of this equilibrium).  Those trajectories starting  on the boundary of $\Xi_{2n}$ with $x^2={\bf{0}}_n$ (including those starting in a neighborhood of the equilibrium) evolve according to the single virus equation, and thus approach the equilibrium. See \cref{ssec:single_virus} above for details. We can further ask whether there is any trajectory starting in a neighborhood of $(\bar x^1, {\bf{0}}_n)$ and in the interior of $\Xi_{2n}$ that might also converge to $(\bar x^1, {\bf{0}}_n)$ (even if some trajectories will not, on account of the instability property). The key conclusion is as follows. (An illustrative example is presented in \Cref{fig:bivirus_perturbation}, but assuming $(0, \bar x^2)$ is the unstable boundary equilibrium.)
  
  \begin{lemma}
  Suppose that the boundary equilibrium $(\bar x^1,{\bf{0}}_n)$ of \eqref{eq:bivirus} is unstable, in the sense that one or more eigenvalues of the associated Jacobian have positive real part. Then there exists no trajectory beginning in the interior of $\tilde \Xi_{2n}$  which approaches this unstable equilibrium, with $\tilde \Xi_{2n}$ is defined above. Namely,  $\tilde \Xi_{2n}$ is equal to $\Xi_{2n}$ with perturbation to encompass the boundary equilibrium $({\bf 0}_n, \bar x^2)$ if it is locally exponentially stable.
  \end{lemma}
  
  \begin{proof}
   Suppose,  to obtain a contradiction, there is a trajectory,  call it $\mathcal T$, starting inside the region $\tilde \Xi_{2n}$ which  has the equilibrium $(\bar x^1,{\bf{0}}_n)$ as its limit. If $({\bf 0}_n, \bar x^2)$ is locally exponentially stable, then any trajectory beginning in $\tilde \Xi_{2n}\setminus\Xi_{2n}$ (i.e. in the perturbed region) is in the region of attraction of $({\bf 0}_n, \bar x^2)$, and thus cannot converge to $(\bar x^1,{\bf{0}}_n)$. Therefore without loss of generality we can assume that the initial condition for the trajectory $\mathcal T$ satisfies $x^2(0)\gg {\bf{0}}_n$. Now for sufficiently large values of time, $t\geq T_0$ say, the trajectory will be arbitrarily close to the limit and thus its evolution from  time $T_0$ onwards can be modelled (to first order) using the linearized equation 
   \begin{equation}\label{eq:linfull}
       \dot x=J(\bar x^1,{\bf{0}}_n)x
   \end{equation}
   which in the light of \eqref{eq:boundaryjacobian} means that the projection $x^2(t)$ satisfies
    \begin{equation}
        \dot x^2=[-D^2+(I-\bar X^1)B^2]x^2
    \end{equation}
 with $x^2(t)\gg 0$ for all finite $t$ and converging to $\vect 0_n$ as $t\to\infty$. 
 
Note that hyperbolicity is required to justify the validity of the linearization approximation and in particular, the drawing of stability conclusions using the linearized equation; hyperbolicity is guaranteed via Theorem~\ref{app:pf_thm_hyperbolicity}.

 Suppose  that  $x^2(T)$ with $T>T_0$ is some point on the projection of the  trajectory $\mathcal T$.  Suppose initially that the eigenvalues of $[-D^2+(I-\bar X^1)B^2]$ are distinct, in which case there are $n$ linearly independent eigenvectors, call them $v_1,v_2,\dots,v_n$. Further suppose that $v_1\gg  {\bf 0}_n$ corresponds to the eigenvalue $\sigma[-D^2+(I-\bar X^1)B^2] $, which is simple since $-D^2+(I-\bar X^1)B^2$ is an irreducible Metzler matrix. Because the equilibrium is unstable, one or more eigenvalues of this matrix must have positive real part (else the entire Jacobian  of \eqref{eq:boundaryjacobian} would have strict left half plane eigenvalues, a property which is guaranteed for its block 11 entry). Hence  the eigenvalue $\sigma[-D^2+(I-\bar X^1)B^2] > 0$, necessarily real and simple by the Metzler property. Write $x^2(T)=\sum_i \zeta_iv_i$ for some scalars $\zeta_i$. Because of the convergence of  $\mathcal T$ to $(\bar x^1, {\bf 0}_n)$, there must hold $\zeta_1=0$, and in fact $\zeta_i=0$ for any $i$ for which the associated eigenvalue of $[-D^2+(I-\bar X^1)B^2]$ has nonnegative real part. Hence $x^2(T)=\sum_{i\in\mathcal I}\zeta_iv_i$, where $\mathcal I$ is the set of indices for which the $i$-th eigenvalue of $[-D^2+(I-\bar X^1)B^2]$ has negative real part. Now suppose $u_1^{\top}$ is the left eigenvector corresponding to eigenvalue $\sigma[-D^2+(I-\bar X^1)B^2]$ with $u_1^{\top}v_1=1$ and $u_1\gg  {\bf 0}_n$. Then because $u_1^{\top}v_i=0$ for $i\neq 1$, it follows that $u_1^{\top}x^2(T)=0$. This is a contradiction to the fact that $x^2(T)\gg {\bf 0}_n$ and the fact that $u_1\gg  {\bf 0}_n$. In other words, no trajectory starting in the interior of $\tilde \Xi_{2n}$ can reach the equilibrium $(\bar x^1, {\bf 0}_n)$. Equivalently, in the vicinity of the equilibrium, all trajectories lying in $\Xi_{2n}$ are pointing away from the equilibrium. 
 
 The case when the eigenvalues of $[-D^2+(I-\bar X^1)B^2]$ are not distinct can be handled by a notationally messy argument involving Jordan blocks; note that the uniqueness of the eigenvalue $\sigma[-D^2+(I-\bar X^1)B^2]$ will be critical. 
  \end{proof}
  
 Despite the replacement of $\Xi_{2n}$ by $\tilde\Xi_{2n}$ to ensure an inward-pointing property for trajectories at the boundary, we cannot apply the Poincar\'e-Hopf Theorem to $\tilde \Xi_{2n}$  because (a) there are trajectories confined to the boundaries and (b) there are zero(s) of the vector field lying on the boundary (corresponding to the always present unstable healthy equilibrium and any boundary equilibria that are unstable). Nor, in an attempt to avoid these problems,  can we apply the Theorem to the interior of $\tilde\Xi_{2n}$ since this is not a compact manifold. We are however in a position to introduce some transformations of the manifold $\tilde\Xi_{2n}$ to resolve these issues. 

\subsubsection{ Behavior in the vicinity of the healthy equilibrium}

 While the healthy equilibrium is unstable, it is very likely to have an associated stable manifold. The following lemma ensures this creates no problem, by showing that the positive orthant is not part of the aforementioned stable manifold. 

 \begin{lemma}
  There exists no trajectory beginning in the interior of $\Xi_{2n}$ which converges to the healthy equilibrium.   
 \end{lemma}

\begin{proof} 
It is obviously enough to establish that the system obtained by linearization about the healthy equilibrium, viz. 
\begin{equation}\label{eq:linearized_healthy}
    \begin{bmatrix} \dot x^1\\\dot x^2\end{bmatrix}=\begin{bmatrix}-D^1+B^1&0\\0&-D^2+B^2\end{bmatrix}\begin{bmatrix}x^1\\x^2\end{bmatrix},
\end{equation}
has the property of the Lemma hypothesis. We examine for convenience just $x^1(t)$. Let $\bar\sigma=\sigma(-D^1+B^1)$, which is  an eigenvalue of $-D^1+B^1$ and positive by \Cref{ass:unstablehealthy} and properties relating $\bar \sigma$ to $\rho((D^1)^{-1}B^1)$ outlined in \cref{ssec:background}. Let $w^{\top}$ be the associated positive left eigenvector with entries summing to 1. With $z=w^{\top}x^1$, there holds $\dot z=\bar \sigma z$, demonstrating that the projection onto the positive vector $w$ of the points on a trajectory of \eqref{eq:linearized_healthy} in the interior of the positive orthant is divergent away from $x^1 = \vect 0_n$. A similar argument establishes the same conclusion that trajectories in the interior of the positive orthant are divergent away from $x^2 = \vect 0_n$  Thus, any trajectory $x(t)$ in the interior of $\Xi_{2n}$ could never converge to the origin $x = \vect 0_{2n}$. 
\end{proof}

 \subsection{Diffeomorphism involving a sphere and use of Poincar\'e-Hopf Theorem}
  
We suppose that, if there is any stable boundary equilibrium (with at most two being able to occur), the region of interest $\Xi_{2n}$ has been perturbed so as to make such equilibria lie in the interior of the perturbed region $\tilde\Xi_{2n}$, as described in \cref{sssec:stable_boundary}. We will introduce diffeomorphic transformations starting with the interior  $\tilde\Xi_{2n}^{\circ}$ of $\tilde\Xi_{2n}$, and subsequently deal with what happens to the boundaries. In so doing, we will be drawing on an argument used by \cite{glass1975combinatorial} for a related but simpler problem (in which $\tilde\Xi_{2n}$ is replaced by the positive orthant and no trajectories confined to boundaries can exist).

A summation of what we are about to prove concerning the interiors of the various regions is in the following Lemma.

\begin{lemma}\label{lem:diffeo}

With notation as previously, there exists a diffeomorphism $F_0$ between the interior $\tilde\Xi_{2n}^{\circ}$ of $\tilde\Xi_{2n}$ and a punctured sphere $S^{2n}\setminus N$, where $N$ denotes the north pole. This diffeomorphism maps the bivirus vector field defined by \eqref{eq:bivirus} on the manifold $\tilde\Xi_{2n}^{\circ}$ onto a unique smooth vector field on $S^{2n}\setminus N$ with zeros which are the images under $F_0$ of the vector field zeros in $\tilde\Xi_{2n}^{\circ}$, and with the indices of corresponding zeros the same. 
\end{lemma}

\begin{proof}
Observe first that the interior of the new region $\tilde\Xi_{2n}$ is obviously diffeomorphic to the interior of a solid ball $\mathcal B_R$ of arbitrary radius $R$ and of dimension $2n$. One could construct such a diffeomorphism, call it $F_1$, by picking a point in the interior of $\Xi_{2n}$ and mapping it to the origin of the ball, and by mapping each line joining that point to a boundary point of $\tilde\Xi_{2n}$ to a line in the same direction joining the origin to the boundary of the ball. We remark that the boundary of $\tilde \Xi_{2n}$ corresponds to the boundary of the ball. 

Next, observe that the interior of the ball is diffeomorphic to $\mathbb R^{2n}$ under the mapping $F_2: \mathcal B_R^{\circ}\to \mathbb R^{2n}, p\mapsto (\tan\frac{\pi\|p\|}{2R})p$. Note that points on the boundary of the ball are mapped to infinity in $\mathbb R^{2n}$.

Under a further diffeomorphism $F_3$, $\mathbb R^{2n}$ corresponds to a sphere $S^{2n}$ excluding one point, say the north pole of the sphere. Such a diffeomorphism is standard, e.g. \cite[see p.6]{guillemin2010differential} or \cite[see p. 29]{ballmann2018introduction}. Points at infinity in $\mathbb R^{2n}$ effectively correspond to the north pole of $S^{2n}$.

The three diffeomorphisms combine to give a single diffeomorphism $F_0=F_3\circ F_2\circ F_1$ from the interior of $\tilde\Xi_{2n}$ to the punctured sphere $S^{2n}\setminus N$ where $N$ denotes the north pole. 

The vector field defining the bivirus equations in \eqref{eq:bivirus}, is, on the manifold $\tilde\Xi_{2n}$, also transformed by $F_0$.  Because the mapping is a diffeomorphism, a unique smooth vector field on $S^{2n}\setminus N$ is guaranteed to exist, \cite[see Proposition 8.19]{lee2013introduction}{\footnote{Two potential difficulties can arise when a vector field on one manifold $\mathcal M$  is mapped to another manifold $\mathcal N$: when the mapping is not surjective, the vector field is not defined at some points of $\mathcal N$, and if the mapping is not injective, one point in $\mathcal N$ may have a nonunique vector field, \cite[see p.181]{lee2013introduction}. The diffeomorphism property of $F_0$ rules out such problems.}}. Further, the index of the vector field at an isolated zero is preserved (together with the isolation of the zero) by a diffeomorphism, \cite[see p.33, Lemma 1]{milnor1997topology}. At a nondegenerate zero, as pointed out by \cite[see p. 136]{guillemin2010differential}, for a diffeomorphism $y=F(x)$, the vector field at $y$ corresponding to $df_x$ at $x$ is given by the similarity transformation $dg_y=(dF_x)^{-1}df_x(dF_x)$, from which it is evident that ${\rm{sign\; det}}(df_{x})={\rm{sign\; det}}(dg_y)$.  Not only is hyperbolicity preserved, but even the actual eigenvalues of the Jacobian at the equilibrium. 
\end{proof}

The entire boundary of $\tilde\Xi_{2n}$ can be identified with the north pole of the sphere $S^{2n}$ under the mapping $F$. Note that this pole amalgamates so to speak the healthy equilibrium (which is unstable), any unstable boundary equilibrium and  trajectories confined to the boundary, and the initial point of trajectories which start on the boundary but immediately leave it (corresponding to i) single-virus behavior of the bivirus system where $x^i(0) = \vect 0_n$ for some $i$, or ii) $x^1(0)+x^2(0) = \vect 1_n$), as such points are a subset of the boundary $\partial \tilde\Xi_{2n}$. On the surface of the sphere (including the north pole) there will be trajectories for which the north pole is a source, and no trajectories will approach the north pole.  This is in the light of our analysis in \cref{ssec:traj_01,ssec:traj_02}, and the properties we established concerning trajectories at and adjacent to the boundary of $\tilde\Xi_{2n}$. This means one can add the north pole to the punctured sphere, and the associated vector field is well defined at the north pole, with a zero there (in fact, this zero is a source). 

Crucially, the sphere itself (without the puncture) is a compact manifold to which the Poincar\'e-Hopf Theorem in principle can be applied. 

\subsection{Completion of proof of Theorem~\ref{thm:main}}

We apply the Poincar\'e-Hopf Theorem to the sphere. The argument is as follows.  Let $n_k$ denote the number of open left half plane eigenvalues of the Jacobian associated with the $k$-th equilibrium or vector field zero in $\tilde\Xi_{2n}^{\circ}$. As noted above in \Cref{lem:diffeo}, this is the same as the number of eigenvalues associated with the Jacobian of the corresponding vector field zero (using the result of \cite{milnor1997topology} and \cite{guillemin2010differential}) on the sphere, any such zero being away from the north pole.   The index for the north pole, corresponding to the boundary of $\tilde\Xi_{2n},$ which is simply a source from the point of view of trajectories on the sphere, is $1$, since there are no left half plane eigenvalues of the Jacobian. Now it is standard that the Euler characteristic of $S^{2n}$ is 2, see e.g. \cite[pg.~134]{hirsch2012differential}. Hence using the Poincar\'e-Hopf Theorem, i.e. \Cref{thm:PH}, for the sphere, we have $\sum_k(-1)^{n_k}+1=2$, or
\begin{equation}
\sum_k (-1)^{n_k}=1
\end{equation}
This equation, although obtained by studying the sphere, is also the equation which relates the vector field zeros for the bivirus problem, with the understanding that the healthy equilibrium is not counted, a boundary equilibrium which is stable (eigenvalues of the Jacobian in the open left half plane) is counted, and a boundary equilibrium which is not stable (one or more eigenvalues of the Jacobian in the right half plane) is not counted. 

\subsection{Consequences of \texorpdfstring{\Cref{thm:main}}{}}
An immediate and important consequence of the main result, \Cref{thm:main}, is the following. It treats all three of the possible configurations of boundary equilibria that can occur. 

\begin{corollary}\label{cor:ph_count}
Adopt the hypotheses of Theorem \ref{thm:main}. Then 
\begin{enumerate} 
\item if both boundary equilibria of the bivirus system are unstable, then there exists an odd number $k\geq 1$ of coexistence equilibria. There are $(k+1)/2\geq 1$ equilibria with the associated Jacobian having an even number of open left-half plane eigenvalues, and $(k-1)/2$ equilibria (none of which can be stable) with the associated Jacobian having an odd number of open left-half plane eigenvalues;
\item 
if both boundary equilibria of the bivirus system are locally exponentially stable, then there exists an odd number $k\geq 1$ of coexistence equilibria. There are $(k+1)/2\geq 1$ equilibria (none of which can be stable) with the associated Jacobian having an odd number of open left-half plane eigenvalues, and $(k-1)/2$ equilibria with the associated Jacobian having an even number of open left-half plane eigenvalues;
\item
if there is one locally exponentially stable and one unstable boundary equilibrium of the bivirus system, there is either no coexistence equilibrium, or an even number $k\geq 2$, of which one half have an associated Jacobian with an even number of open left-half plane eigenvalues, and one half (none of which can be stable) have an associated Jacobian with an odd number of open left-half plane eigenvalues.
\end{enumerate}
\end{corollary}

\begin{proof}
For the first claim, suppose there are two unstable boundary equilibria. Since we do not count any unstable boundary equilibria in computing \eqref{eq:count}, there must be at least one equilibrium in $\tilde \Xi_{2n}^\circ$ contributing to the sum on the left hand side of \eqref{eq:count} in order that the sum be positive, and the associated Jacobian must have an even number of eigenvalues in the open left-half plane. Thus $k \geq 1$. Suppose that $k_e,(k_0) $ denote the number of equilibria with an even (odd) count of open left-half plane eigenvalues, with $k_e+k_0=k$.  Since an equilibrium with a Jacobian having an even (respectively, odd) number of eigenvalues in the open left-half plane contribute $+1$ (respectively, $-1$) to the right hand side of \eqref{eq:count}, there holds $k_e-k_0=1$.  The remaining conclusions of item 1) of the corollary are easily established. 

For the second claim, suppose  there are two locally exponentially stable boundary equilibria. The argument is the same as that for the first claim, save that \eqref{eq:count} yields $k_e-k_0=-1$, in light of the two stable boundary equilibria.

For the third claim, let us note the fact that a configuration of one stable and one unstable boundary equilibria is possible has been known in the literature for some time, e.g.~\cite[Corollary~3.11]{ye2021_bivirus} and \cite{santos2015bi}. Given the stated equilibrium pattern and with no coexistence equilibria, there is a single summand in the sum \eqref{eq:count}, associated with the single stable boundary equilibrium. The claim covering the case when there are coexistence equilibria is proved in like manner to the first and second claims. (An example below will demonstrate that such a configuration of boundary equilibria can also allow for the presence of coexistence equilibria, i.e., $k \geq 2$.) 
\end{proof}

If the bivirus system has two unstable boundary equilibria (Item~2 of \Cref{cor:ph_count}), then one can further exploit known properties of monotone systems to conclude that among the $k$ coexistence equilibria, at least one of them is locally exponentially stable~\cite[Theorem~2.8]{smith1988monotone_survey}.
We will develop further counting conditions in the next section, based on Morse inequalities for Morse-Smale systems, which provides an alternative method to show that there is necessarily a stable coexistence equilibrium if both boundary equilibria are unstable.
A weaker version of the second claim of the lemma can be found in \cite{janson2020networked}, where it is established that there must be at least one coexistence equilibrium (but no stability properties are provided for the equilibrium).

\subsection{Application of results to example in \texorpdfstring{\Cref{ssec:n4_example}}{}}

Recall that it was possible using simulations of the $n=4$ example in \Cref{ssec:n4_example} to establish that there were two stable boundary equilibria and two stable coexistence equilibria, but determining the number of unstable coexistence equilibria was more involved. With the results obtained using the Poincar\'e-Hopf Theorem (i.e. \Cref{thm:main} and \Cref{cor:ph_count}), we can easily obtain a lower bound on the number of unstable coexistence equilibria based on knowledge of the presence of the stable equilibria. Evidently, the two stable coexistence equilibria have an even number of open left-half plane eigenvalues (being 8, the system dimension). Thus, the total number of coexistence equilibria with an even number of open left-half plane eigenvalues must be $(k-1)/2 \geq 2$, which implies $k \geq 5$. In other words, there are at least $5$ coexistence equilibria, of which at least $3$ must be unstable with an odd number of eigenvalues in the open left-half plane. Not only does this underscore the complexity of the equilibria patterns for networked bivirus SIS systems, but highlights the additional insights provided through Poincar\'e-Hopf Theory. While monotone systems theory (\cite[Proposition~2.9]{smith1988monotone_survey}) can allow one to conclude the presence of $3$ unstable coexistence equilibria, the eigenvalue properties cannot be so obtained.

\section{Further counting results, involving inequalities}\label{sec:morse_smale}

Well after the original work establishing the Poincar\'e-Hopf formula of \cref{thm:PH}~\cite{hopf27}, further counting results were obtained involving the equilibria of an equation $\dot x=f(x)$ defined on a compact manifold $\mathcal M$, which were additional to (though also incorporating) that formula. We briefly summarize below those aspects of the results needed for use on a sphere, and then demonstrate its application to the bivirus system. 

Two major sequential developments provided the results. The first built on the work of Morse \cite{morse1934calculus} on critical points of a smooth scalar function, call it $g(x)$, defined on a $n$-dimensional manifold. For a summary, see~\cite[pp. 28-31]{milnor2016morse} and \cite[pp. 290-291]{mukherjee2015differential}, while \cite{matsumoto2002introduction} contains a more leisurely treatment. On an $n$-dimensional manifold, the nondegenerate critical points of a scalar function (nondegenerate critical points being those where the gradient is zero and the Hessian is nonsingular) may be minima, maxima, or saddle points, corresponding to the number of negative eigenvalues of the Hessian being $0$, $n$ or any integer in between, respectively. The number of such eigenvalues is termed the Morse index.  Morse obtained a set of $n+1$ inequalities (including one equality) relating the numbers of saddle points with different Morse indices, assuming that the number of critical points is finite and all are nondegenerate; the set of inequalities also involves the values of certain topological indices termed Betti numbers, see e.g. \cite{matsumoto2002introduction}, in addition to the Euler characteristic of the manifold.  

The second and further major advance on this work  can be attributed to Smale, who studied the equilibria of systems $\dot x=f(x)$ defined on a manifold $\mathcal M$; the work essentially gave identical results to Morse Theory in the special case when $f(x)={\mbox{grad}}_{\mathcal M}g(x)$ is a gradient  of some smooth scalar function $g(x)$, given that further conditions are imposed on $f(x)$, as described further below. One such restriction is that all equilibria are hyperbolic.   For introductory remarks on such systems, see \cite{zomorodian2005topology}, while the key reference for our use of such ideas is~\cite{smale1967differentiable}. For an $n$-dimensional manifold, let $c_{\lambda}$ denote the number of equilibrium points for which the associated Jacobian has precisely $\lambda$ eigenvalues with negative real part. Then a set of inequalities involving the $c_{\lambda}$ can written down, which include an equality that is equivalent to the equality arising in the Poincar\'e-Hopf Theorem. The Euler characteristic and the Betti numbers of the manifold appear in the inequalities. More details are now offered relevant to their application to the bivirus problem.  

\subsection{Definition and Properties of Morse-Smale systems}

The definition of a Morse-Smale system requires an understanding of the concepts of the stable manifold and unstable manifold of an equilibrium point of a dynamical system \cite{smale1960morse,sastry2013nonlinear, wiggins2003introduction}. Roughly speaking, the stable manifold of an equilibrium point is the set of points from which forward-time trajectories will converge to the point, and the unstable manifold is the set of points from which backward-time trajectories will converge to the point. While such sets do not always constitute manifolds, they do so when equilibrium points are hyperbolic \cite[pp. 289-290]{sastry2013nonlinear}. The dimension of a stable (unstable) manifold is then the number of left (right) half plane eigenvalues of the Jacobian $df_x$ evaluated at the equilibrium. Note that a hyperbolic equilibrium point which is not stable will have an associated stable manifold unless it is actually a source (i.e. the associated Jacobian matrix has all eigenvalues with positive real parts). The following defines the properties characterizing a Morse-Smale dynamical system.

\begin{definition}\label{ass:MS}
A smooth dynamical system $\dot x=f(x)$ existing on some $n$-dimensional manifold $\mathcal M$ is a \textit{Morse-Smale system} when the following conditions hold:
\begin{enumerate}
\item
 Trajectories have no finite escape times in the forward or backward directions,  i.e. $\|x(t)\| \to \infty$ when $t\to T$ for some finite $T$ is not possible, and for any initial condition, solutions exist in both directions  on $(-\infty,\infty)$.
 \item
 Equilibrium points are hyperbolic (i.e. at an equilibrium point $x_k$, the matrix $df_{x_k}$ has no eigenvalues with zero real part).
 \item
If the stable manifold $\mathcal W_s(x_j)$ of the equilibrium point $x_j$ intersects the unstable manifold $\mathcal W_u(x_k)$ of a second equilibrium point $x_k$, the intersection is transverse, that is, if $x\in\mathcal W_s(x_j)\cap\mathcal W_u(x_k)$, then the span of the corresponding tangent space is $\mathbb R^n$, i.e. $T_x(\mathcal W_s(x_j))+T_x(\mathcal W_u(x_k))=\mathbb R^n$. This condition may alternatively be stated as ${\rm{dim}} \mathcal W_s(x_j)+{\rm{dim}}W_u(x_k)-n={\rm{dim}}(T_x(\mathcal W_s(x_j))\cap T_x(\mathcal W_u(x_k)))$.
\item
If there are periodic orbits, they are hyperbolic.\footnote{For an explanation of hyperbolicity of periodic orbits (which is a generalization of the idea of hyperbolicity of an equilibrium point), see e.g. \cite{robinson1998dynamical}. This paper however will make virtually no use of this notion.} In particular then, nonattractive limit cycles are not permitted.
\end{enumerate}
\end{definition}

The key counting result for general Morse-Smale systems, though simplified by exclusion of the possibility of periodic orbits (since such an exclusion will be justified in applying the result to bivirus sytems), is presented below, first in general form and then  specialised to the case of motion on a sphere $S^{2n}$:

\begin{theorem}\label{thm:morsesmale}
Suppose that $\dot x=f(x)$ is a Morse-Smale system without limit cycles defined on an $n$-dimensional compact manifold $\mathcal M$. Let $c_{\lambda}$ denote the number of equilibrium points\footnote{The fact that $c_{\lambda}$ must be defined using a particular choice of coordinate basis at an equilibrium point but assumes a value that is independent of the choice is not explicitly demonstrated in \cite{smale1967differentiable} but is implicitly assumed.} for which the associated Jacobian has precisely $\lambda$ eigenvalues with negative real part, or equivalently the associated stable manifold has dimension $\lambda$.
Let $r_{\lambda}$ denote the rank of the $\lambda$-th homology group of $\mathcal M$ (the $\lambda$-th Betti number).  Then the following inequalities  hold:
\begin{align}\label{eq:morsesmaleineq}
 c_0   &\geq r_0  \\\notag
   c_1-c_0  &\geq r_1-r_0 \\\notag
  c_2-c_1+c_0 &\geq r_2-r_1+r_0 \\\notag
            &\vdots\\\notag
c_{n-1}-c_{n-2}+\cdots+(-1)^{n-1}c_0&
\geq r_{n-1}-r_{n-2}+\cdots+(-1)^{n-1}r_0\\\notag
c_n-c_{n-1}+\cdots+(-1)^nc_0&=r_n-r_{n-1}+\cdots+(-1)^nr_0=(-1)^n\chi(\mathcal M)
\end{align}
\end{theorem}

The last equation above can be rewritten as 
\[
(-1)^kc_k=\chi(M)
\]
This is the same as the equality \eqref{eq:PHmainequation} resulting from the Poincar\'e-Hopf Theorem. To see this, recall that any equilibrium whose Jacobian has an odd number of eigenvalues with negative real part has a negative index i.e. the sign of the determinant of the Jacobian is negative. Hence the left hand side of the last equation adds together with correct sign the indices of all the equilibrium points, and  is simply $\sum_k{\mbox{ind}}_{x_k}(f)$, which from \eqref{eq:PHmainequation} is $\chi(M)$.

For a sphere $S^n$, the only nonzero homology groups are $H_0=H_n=\mathbb Z$, and so $r_0=r_n=1$, but otherwise $r_{\lambda}=0$. Further $\chi(S^n)=1+(-1)^n$. These properties are set out in \cite[see p.141]{matsumoto2002introduction}. For our purposes, we record what happens for the even dimension sphere $S^{2n}$:
\begin{corollary}
Adopt the same hypotheses as for Theorem \ref{thm:morsesmale}, save regarding the dimension of $\mathcal M$, and suppose that $\mathcal M$ is $S^{2n}$. Then there holds
\begin{align}\label{eq:MSsphere}
    c_0&\geq 1\\\notag
    c_1-c_0&\geq -1\\\notag
    c_2-c_1+c_0&\geq 1\\\notag&\vdots\\\notag
    c_{2n-1}-c_{2n-2}+\dots-c_0&\geq -1\\\notag
    c_{2n}-c_{2n-1}+\dots-c_1+c_0&=2
\end{align}
\end{corollary}

\subsection{Application to the bivirus problem}

We now indicate how these inequalities affect the bivirus equation results. We actually  apply them to the system obtained by transforming the bivirus equations in the region $\tilde\Xi_{2n}$ to the sphere $S^{2n}$. Such a development must rest on an assumption the bivirus system is Morse-Smale. However,  as far as the authors are aware, there are no formal results which establish that the property holds for a bivirus system, and this paper does not provide an explicit proof. Rather, as we now argue, the known properties of the bivirus system makes it \textit{reasonable to assume} that it is Morse-Smale, and we do so for the purposes of advancing our counting approach.

Recall that boundary points of the bivirus equations correspond to the north pole of the sphere, which is a source. Condition 1 of the Morse-Smale system definition is trivially fulfilled. Condition 2 is effectively covered by genericity of $D^i$ and $B^i$, see Theorem \ref{app:pf_thm_hyperbolicity}. Formal demonstration of Condition 3 is however \textit{not} possible. However, Condition 3 also appears as if it is generically satisfied.  We note that it is well understood that smooth dynamical systems defined on a manifold are generic in a particular sense, see~\cite{smale1967differentiable}. Genericity here actually refers to the notion of possibility perturbing the vector field in a small region by arbitrarily small bumps, rather than changing the numerical values of the parameters appearing in the differential equations.

Interpreted for the bivirus system as opposed to the system defined on the sphere, the genericity results mean that the precise models containing for example quadratic terms in the state within the vector field might not be Morse-Smale, though an arbitrarily small perturbation will be Morse-Smale.  We also note that the reference~\cite{robbin1981algebraic}  proves that polynomial vector fields of a given degree are generically Morse-Smale. In the case of bivirus systems, despite the fact that the matrices $D^i,B^i$ are generic, the associated system is however not generic within the set of \textit{all} quadratic vector fields: observe that the component of the vector field associated with, for example, $\dot x^1_i$ includes $x^1_i\beta^1_{ij}x^1_j$ and $x^2_i\beta^1_{ij}x^1_j$, i.e two of the quadratic terms have the same coefficient. Apart from this, allowing $B^i$ with zero entries (albeit with an irreducibility assumption) is also a form of specialization moving the system outside the scope of those covered in \cite{robbin1981algebraic}.

As for the fourth requirement characterising a Morse-Smale system, recall that because the bivirus system is a monotone system, the only limit cycles possible are those which are nonattractive, see \cite[see p.95] {smith1988monotone_survey}.   However, the mere occurrence of any type of limit cycle in a bivirus model appears to be a nongeneric property:  limit cycles, nonattractive or otherwise, have never been observed in the bivirus literature. Hence we will assume that for a generic bivirus system (and its mapping onto the sphere $S^{2n}$), there are no limit cycles of any type.

\subsection{Additional insights for the bivirus problem}

Previously, we showed how a counting result involving equilibria for the bivirus problem could be obtained by relating the problem to trajectories on a sphere and appealing to the Poincar\'e-Hopf formula. The key adjustment was to use a single source equilibrium on the sphere to account for the healthy equilibrium in the bivirus problem, together with any unstable boundary equilibrium, of which there can be zero, one or two. Stable equilibria on the sphere include those corresponding to stable boundary equilibria for the bivirus problem, of which there can be two, one or zero (corresponding to zero, one or two unstable boundary equilibria). 

We sum up the second main counting result of the paper, flowing from Morse-Smale theory and which we have just proved, as follows.
\begin{theorem}\label{thm:bivirusMS}
With notation as previously defined, consider the equation set \eqref{eq:bivirus} and suppose that Assumptions \ref{ass:constraints} and  \ref{ass:unstablehealthy} hold. Suppose further that the four conditions in \Cref{ass:MS} all hold and hence \eqref{eq:bivirus} is a Morse-Smale system, thereby guaranteeing that equilibria of the equations in the region of interest $\Xi_{2n}$ are all nondegenerate and thus finite in number, and hyperbolic. 
Let $c_{\lambda}$ denote the number of equilibria in the region of interest whose associated Jacobian has $\lambda$ open left-half plane eigenvalues.
The healthy equilibrium $({\bf{0}}_n,{\bf{0}}_n)$ and any unstable boundary equilibrium together contribute an allowance of 1 to $c_0$, and any coexisting source makes a further contribution of 1. Then the equation set \eqref{eq:MSsphere} holds. 
\end{theorem}

Using these counting conditions, we can obtain further insights into the nature of the equilibria.

\begin{corollary}\label{cor:MScount}
Adopt the hypotheses of Theorem \ref{thm:bivirusMS}. Then 
\begin{enumerate}
    \item 
    There always exists a stable equilibrium.
    \item
    If there are two unstable boundary equilibria there is a stable coexistence equilibrium.
    \item
    If every inequality in the set \eqref{eq:MSsphere} is an equality, then $c_0=1,c_1=c_2=\dots=c_{2n-1}=0,c_{2n}=1$  and conversely, and the only corresponding equilibrium configurations are one stable and one unstable boundary equilibrium with no interior equilibrium, or two unstable boundary equilibrium and one interior stable equilibrium; both configurations are possible. 
\end{enumerate}
\end{corollary}

\begin{proof}
To establish the first claim, observe that addition of the last two equations in \eqref{eq:MSsphere} yields $c_{2n}\geq 1$, implying there is at least one stable equilibrium. The second claim is a consequence of the first, and the allowed patterns of stability for the boundary equilibria. 

For the third claim, the values of $c_i$ are trivial to establish if the inequalities are in fact all equalities. The consequential configuration restrictions are immediate. 
\end{proof}

Examples illustrating the above claims are provided in the sequel.

We can easily make some specific remarks applying to the case $n=2$.
It is shown in \cite{ye2021_bivirus} that in addition to the two boundary equilibria, there can be zero, one or two coexistence equilibria.

Unstable boundary equilibria make no difference to the value of $c_0$ while stable boundary equilibria add to the value of $c_4$. Coexistence equilibria may add to the value of any $c_i$. 

We now set out the full range of possibilities for different types of equilibria.

\begin{enumerate}
    \item 
    Suppose there are no coexistence equilibria. Necessarily, $c_0=1$, corresponding to the healthy equilibrium, and any unstable boundary equilibria. However, with no coexistence equilibria,  the first claim of Corollary \ref{cor:MScount} then implies there must be at least one stable (and could be two) boundary equilibrium, corresponding to $c_4$ being $1$ or $2$. It is easily verified that only the value $c_4=1$ is consistent with \eqref{eq:MSsphere}. Thus one boundary equilibrium is stable and the other is unstable. (This conclusion is also available in \cite[see Corollary 3.16]{ye2021_bivirus}.) There holds $\{c_0,c_1,c_2,c_3,c_4\}=\{1,0,0,0,1\}$. 
    \item
    Suppose there is one coexistence equilibrium.  First, this is consistent with there being two unstable boundary equilibria and a stable coexistence equilibrium, i.e. $\{c_0,c_1,c_2,c_3,c_4\}=\{1,0,0,0,1\}$. There are no other possibilities with two unstable boundary equilibria.  Second, it is also consistent with there being two stable boundary equilibria, and one coexistence equilibrium which is unstable with precisely one unstable eigenvalue for its Jacobian, implying  $\{c_0,c_1,c_2,c_3,c_4\}=\{1,0,0,1,2\}$. There are no other possibilities with two stable boundary equilibria. There are no possibilities at all with one stable and one unstable boundary equilibrium. Conversely, if there are two stable boundary equilibria, or two unstable boundary equilibria, there is precisely one coexistence equilibrium.  Thus the stability properties of the single coexistence equilibrium are governed by the stability properties of the boundary equilibria, which must be both stable or both unstable. 
    \item
    Now suppose there are two coexistence equilibria. By the preceding point, there is necessarily one stable and one unstable boundary equilibrium.  One can  check the various possibilities to conclude that the following possibilities exhaust those which are consistent with \eqref{eq:MSsphere}
    \begin{align*}
        \{c_0,c_1,c_2,c_3,c_4\}\in \{\{1,0,0,1,2\},\{2,1,0,0,1\},\{1,1,1,0,1\},\{1,0,1,1,1\} \}
    \end{align*}
    The first possibility corresponds to one stable coexistence equilibrium, and the second coexistence equilibrium having a single positive eigenvalue of its Jacobian. The second possibility corresponds to two unstable coexistence equilibria, with one being a source, and the other having three unstable (positive real part) Jacobian eigenvalues. The third possibility corresponds to both coexistence equilibrium being unstable, with three and two unstable eigenvalues of the Jacobian. The last possibility captures a situation with two unstable coexistence equilibria, with one and two unstable eigenvalues of the Jacobian.
\end{enumerate}

\section{Numerical examples}\label{sec:examples}
We now use two numerical examples to illustrate some complex equilibria patterns and the Morse inequalities developed in \Cref{thm:bivirusMS}.

\subsection{Example 1} We return to the $n=4$ example presented in \Cref{ssec:n4_example}, recalling that we concluded below \Cref{cor:ph_count} there exist at least 5 coexistence equilibria (two being stable and identified in \Cref{ssec:n4_example}, and at least three being unstable).

We constructed this example by taking two separate $n=2$ systems (in which case it is possible to analytically determine all equilibria~\cite{ye2021_bivirus}) and weakly coupling them together via the $1\times 10^{-3}$ terms in $B^1$ and $B^2$ --- more precisely, via a homotopy. One can either use a numerical solver, or a gradient descent algorithm to locate the \textit{unstable} coexistence equilibria. In the latter approach, one must start the algorithm sufficiently close to an unstable equilibrium --- our knowledge that one expects isolated and hyperbolic equilibria, viz. \Cref{thm:finiteness,app:pf_thm_hyperbolicity}, establishes that the joined $n=4$ system will have equilibria which are perturbations of the separate $n=2$ systems, providing the critical starting location information. The technical details are beyond the scope of this paper, and thus omitted. 

Using such a gradient descent algorithm, we located 5 unstable coexistence equilibria $(\tilde x^1, \tilde x^2)$ which are given by
\begin{align*}
    &\left(\begin{bmatrix}
    0.3478\\
    0.2662\\
    0.2298\\
    0.3899
    \end{bmatrix},\begin{bmatrix}
    0.2298\\
    0.3899\\
    0.3478\\
    0.2662
    \end{bmatrix}\right),\;
    \left(\begin{bmatrix}
    0.3574\\
    0.2761\\
    0.0039\\
    0.0084
    \end{bmatrix}, \begin{bmatrix}
    0.2211\\
    0.3785\\
    0.6109\\
    0.6079
    \end{bmatrix}\right),\;
    \left(\begin{bmatrix}
    0.0055\\
    0.0032\\
    0.2388 \\
    0.4016
    \end{bmatrix}, 
     \begin{bmatrix}
    0.5596 \\
    0.7119 \\
    0.3379\\
    0.2563
    \end{bmatrix}\right),\nonumber \\
    &\left(\begin{bmatrix}
    0.3379\\
    0.2563\\
    0.5596\\
    0.7119\\
    \end{bmatrix}, 
     \begin{bmatrix}
    0.2388\\
    0.4016\\
    0.0055\\
    0.0032
    \end{bmatrix}\right),\;
    \left(\begin{bmatrix}
    0.6109\\
    0.6079\\
    0.2211\\
    0.3785\\
    \end{bmatrix}, 
     \begin{bmatrix}
    0.0039\\
    0.0084\\
    0.3574\\
    0.2761
    \end{bmatrix}\right).
\end{align*}
The Jacobian of the first listed equilibrium has two unstable eigenvalues, while the Jacobian of all other equilibria have just one unstable eigenvalue.

In terms of the Morse inequalities for the $8$ dimensional system, we thus have $c_0 = 1$ (the healthy equilibrium), $c_8 = 4$ (the two stable boundary equilibria and two stable coexistence equilibria from \Cref{ssec:n4_example}), $c_7 = 4$, $c_6 = 1$, and $c_i = 0$ for $i = 1,2,\hdots ,5$. It is easily verified that all inequalities (and the final equality) in \eqref{eq:MSsphere} hold.

\subsection{Example 2} We next present another $n=4$ example differing from that presented in \Cref{ssec:n4_example}, with $D^1 = D^2 = I$ and 
\begin{equation}
    B^1=\begin{bmatrix} 1.6&1& 0.001 & 0.001 \\ 1&1.6 & 0.001& 0.001 \\ 0.001 & 0.001 & 1.7&1\\0.001& 0.001 & 1.2&0.5
    \end{bmatrix},
    \quad
    B^2=\begin{bmatrix} 2.1&0.156& 0.001 & 0.001 \\ 3.0659&1.1 & 0.001& 0.001 \\ 0.001 & 0.001 & 1.6&1\\0.001& 0.001 & 1.2&0
    \end{bmatrix}.
\end{equation}

The boundary equilibria $(\bar x^1, \vect 0_4)$ and $(\vect 0_4, \bar x^2)$ are locally exponentially stable and unstable, respectively, and given by:
\begin{equation}
    \left(\begin{bmatrix}
    0.6158\\
    0.6155\\
    0.6030\\
    0.4927
    \end{bmatrix},\vect 0_4\right),\qquad \left(\vect 0_4,\begin{bmatrix}
    0.5651\\
    0.7163\\
    0.5683\\
    0.4059
    \end{bmatrix}\right).
\end{equation}
There are two coexistence equilibria given by
\begin{align*}
    &\left(\begin{bmatrix}
    0.0095\\
    0.0056\\
    0.5965\\
    0.4875
    \end{bmatrix},\begin{bmatrix}
    0.5555\\
    0.7089\\
    0.0062\\
    0.0044
    \end{bmatrix}\right),\qquad
    \left(\begin{bmatrix}
    0.3391\\
    0.2576\\
    0.5998\\
    0.4901
    \end{bmatrix}, \begin{bmatrix}
    0.2376\\
    0.4001\\
    0.0031\\
    0.0022
    \end{bmatrix}\right).
\end{align*}
The former is locally exponentially stable, while the latter is unstable and its Jacobian has a single unstable eigenvalue. Sample trajectories for convergence to the stable boundary equilibrium and stable coexistence equilibrium are given in \Cref{fig:example2}.

\begin{figure*}[!htp]
\begin{minipage}{0.475\linewidth}
\centering
\subfloat[Boundary equilibrium with virus~1 endemic\label{fig:n4_example2_boundary}]{\includegraphics[width=\columnwidth]{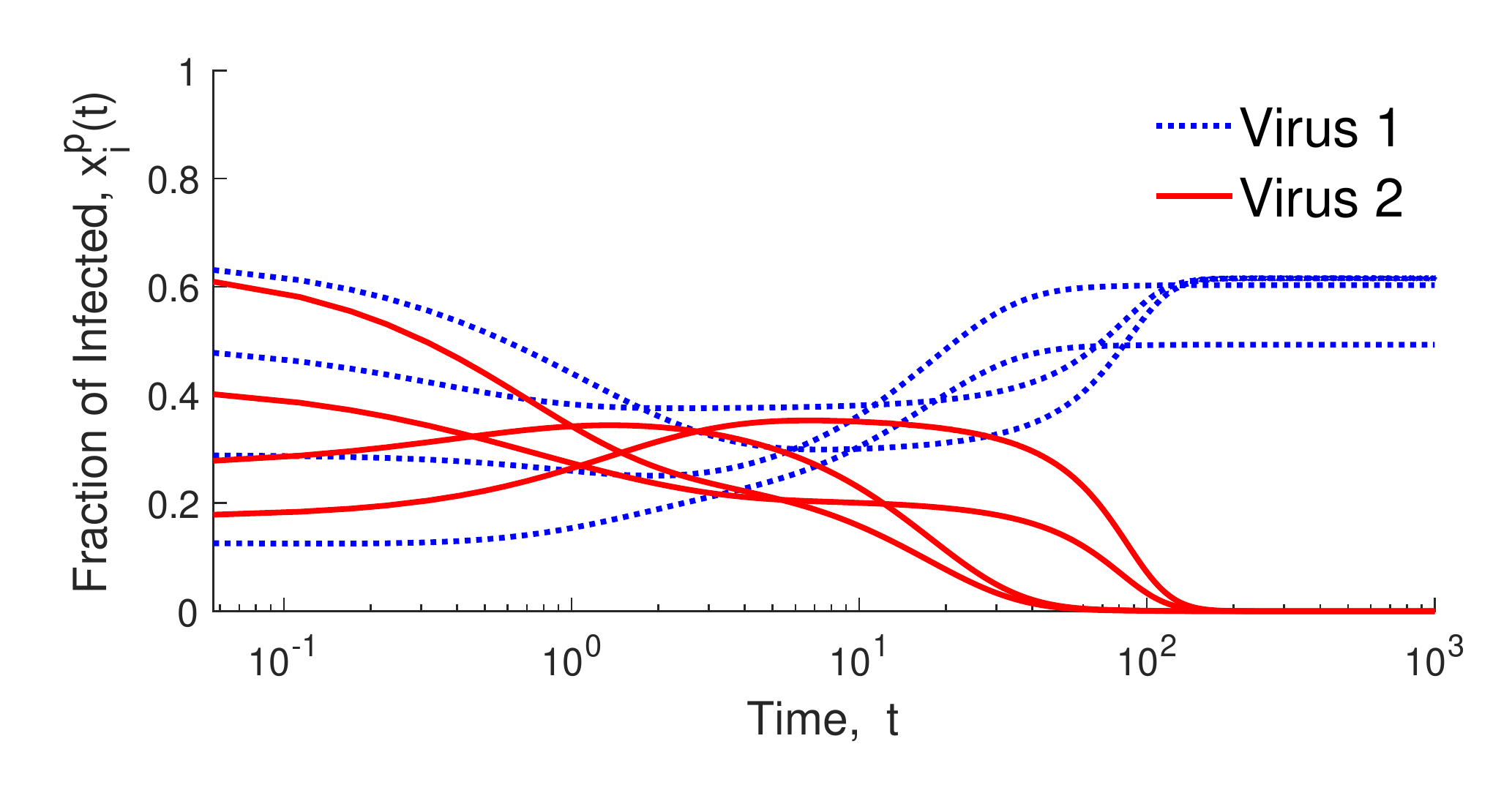}}
\end{minipage}
\hfill
\begin{minipage}{0.475\linewidth}
\centering\subfloat[Stable coexistence equilibrium]{\includegraphics[width=\columnwidth]{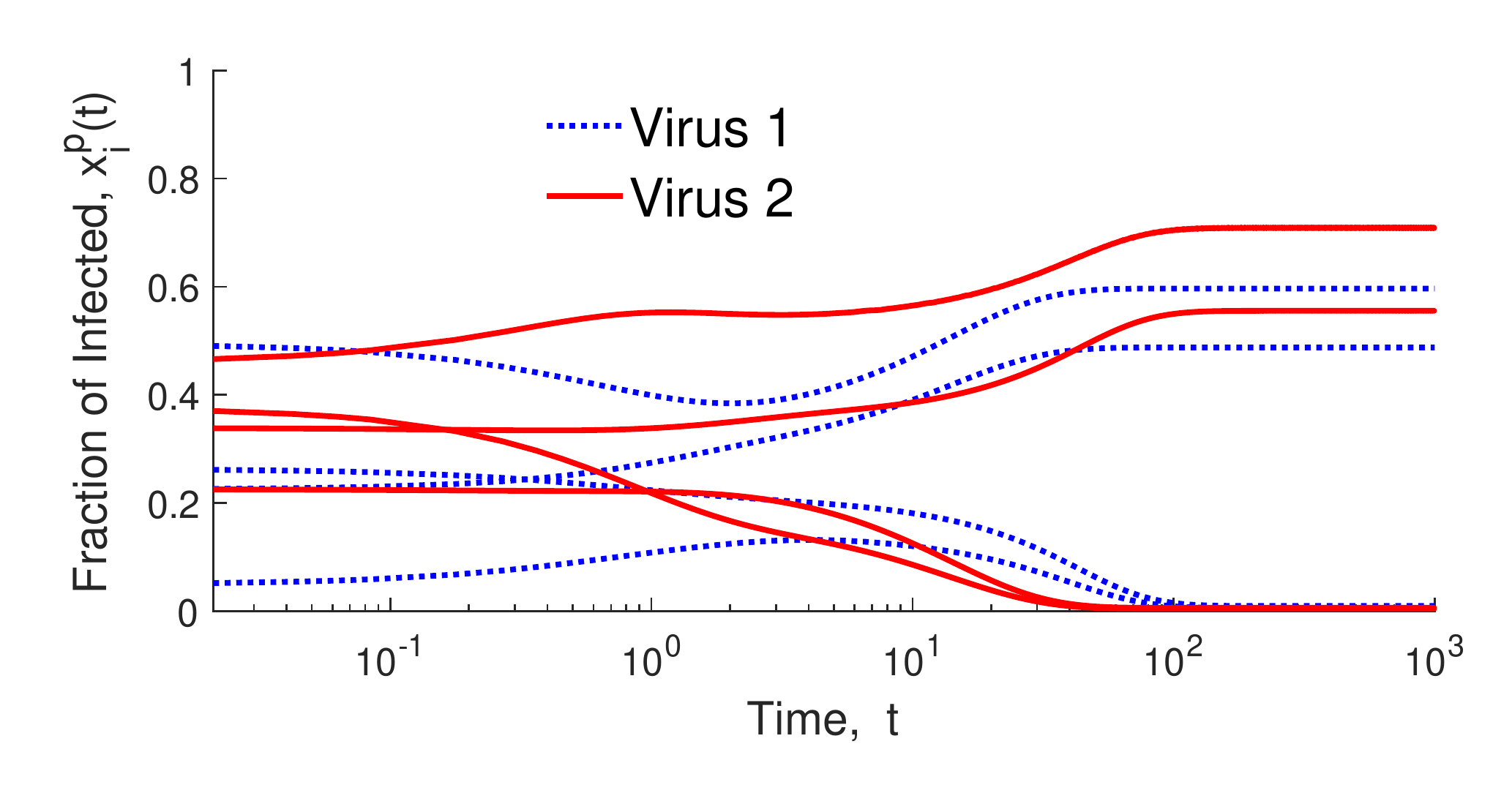}\label{fig:n4_example2_coexistence}}
\end{minipage}
\caption{Sample trajectories of \textit{Example~2} system, with different initial conditions. As is evident, convergence occurs to two different attractive (locally exponentially stable) equilibria, dependent on the initial conditions.  }\label{fig:example2}
\end{figure*}

In terms of the Morse inequalities for the $8$ dimensional system, we thus have $c_0 = 1$ (the healthy equilibrium and the unstable boundary equilibrium $(\vect 0_4, \bar x^2)$), $c_8 = 2$ (the stable boundary equilibrium $(\bar x^1, \vect 0_4)$ and the stable coexistence equilibrium), $c_7 = 1$ (the unstable coexistence equilibrium), and $c_i = 0$ for $i = 1,2,\hdots ,6$. Again, all inequalities (and the final equality) in \eqref{eq:MSsphere} hold.

We conclude by remarking that a recent preprint provided a numerical example of a bivirus system (modified to have additional nonlinearities in the dynamics) with multiple attractive coexistence equilibria~\cite{doshi2022convergence_bivirus}. To the best of our knowledge, our work and \cite{doshi2022convergence_bivirus} are the first to demonstrate multiple coexistence equilibria for networked bivirus models. However, \cite{doshi2022convergence_bivirus} only provides a single numerical example and is limited to simulations of convergence to stable coexistence equilibria. Here, we provide significant theoretical advances that establish counting results on how many coexistence equilibria there may be, and their stability properties (including unstable equilibria and the number of unstable eigenvalues of their Jacobian).

 \section{Conclusions}\label{sec:conclusions}
In this paper, we applied the Poincar\'e-Hopf Theorem to the SIS networked bivirus model, which required significant adaptation due to various complexities of the bivirus dynamics. We then applied Morse inequalities, under the assumption that the bivirus system was a Morse-Smale system. Through these methods, we obtained a set of counting results which, given different stability configurations of boundary equilibria, yield lower bounds on the number of coexistence equilibria, and importantly, information about the number of stable eigenvalues of their Jacobian matrices. We provided numerical examples to illustrate the results, and provide evidence of the highly complex coexistence equilibria patterns possible. In future work, we aim to extend our approach to multivirus models with three or more competing viruses, and identify explicit relations between the parameter matrices $D^i, B^i$ and the number of coexistence equilibria.

\appendix

\section{Proof of \texorpdfstring{ \Cref{thm:finiteness}}{}}\label{app:pf_thm_finiteness}

The introduction below of the Parametric Transversality Theorem, the principal tool in the proof, requires the preliminary calculation of certain Jacobians. We will first consider the case where the $B^i$ are fixed, to demonstrate that for almost all allowed $D^i$, equilibria are nondegenerate and then the case with $D^i$ fixed and $B^i$ variable. Finiteness of the number of zeros will follow straightforwardly from the nondegeneracy conclusion. 

Without loss of generality, we will assume there is some fixed positive $\bar d<1$ such that all diagonal entries of the $D^i$ lie in $(\bar d,\bar d^{-1})$. We will separately consider equilibrium points in the interior of $\Xi_{2n}$ and those on its boundary, beginning with the former. Let $\mathcal X,\mathcal D$ denote manifolds which are respectively the interior of $\Xi_{2n}$ (and thus an open set) and the set of allowed $D^i$ consistent with the choice of $\bar d$. Consider for any $x=[(x^1)^{\top},(x^2)^{\top}]^{\top}\in\mathcal X$ and $\delta=\mbox{vec}[D^1,D^2]\in\mathcal D$ the map
\[
f_{\delta}:\mathcal X\times\mathcal D\to\mathcal Y, (x,\delta)\mapsto y=\begin{bmatrix}[-D^1+(I_n-X^1-X^2)B^1]x^1\\
[-D^2+(I_n-X^1-X^2)B^2]x^2
\end{bmatrix}
\]

\begin{lemma}
With notation as above, the matrix $\frac{\partial f_{\delta}(x,\delta)}{\partial (x,\delta)}$ has full row rank at any coexistence equilibrium. 
\end{lemma}
\begin{proof}
We first observe that
\begin{small}
\begin{equation}\label{eq:jacxdelta}
    \frac{\partial f_{\delta}(x,\delta)}{\partial (x,\delta)}=\frac{\partial f_{\delta}(x^1,x^2,D^1,D^2)}{\partial (x^1,x^2,D^1,D^2)}=\left[\frac{\partial f_{\delta}(x^1,x^2,D^1,D^2)}{\partial (x^1,x^2)}\;\frac{\partial f_{\delta}(x^1,x^2,D^1,D^2)}{\partial (D^1,D^2)}\right]
\end{equation}
\end{small}
Observe next, using the diagonal nature of the $D^i$, that there holds (with $e_j$ denoting the $j$-th unit vector)
\begin{equation}
\frac{\partial f_{\delta}(x^1,x^2,D^1,D^2)}{\partial \delta^1_j}=x^1_je_j, \frac{\partial f_{\delta}(x^1,x^2,D^1,D^2)}{\partial \delta^2_j}=x^2_je_{j+n}
\end{equation}
It follows that 
\begin{equation}
     \frac{\partial f_{\delta}(x^1,x^2,D^1,D^2)}{\partial \delta}=\frac{\partial f_{\delta}(x^1,x^2,D^1,D^2)}{\partial (D^1,D^2)}=\begin{bmatrix}
    X^1&0\\0&X^2
    \end{bmatrix}
\end{equation}

At any equilibrium in $\Xi_{2n}$, this matrix has full row rank, as then does the Jacobian matrix $\frac{\partial f_{\delta}(x,\delta)}{\partial (x,\delta)}$ in \eqref{eq:jacxdelta} of which it is a submatrix. 
\end{proof}

Now, we carry out a similar calculation involving the $B^i$.

There are a finite number of patterns of zeros in the $B^i$ consistent with their being irreducible; this pattern can be associated with the edge set $\mathcal{E}$ defining the network structure for $\mathcal{G}^i$, while the nonzero values in $B^i$ define the edge weights. We will assume a fixed but arbitrary pattern in this set and prove that, excluding a zero measure set of the nonzero values, the desired result holds. Note that whatever the set is, in every row of $B^1$ and $B^2$ there is at least one positive entry,  a consequence of the irreducibility property (and equivalently the strong connectedness property of $\mathcal{G}^1$ and $\mathcal{G}^2$).  

Without loss of generality, we will also assume that there is some arbitrarily large but fixed real $\bar b$ such that  there always holds $\beta^1_{ij}<\bar b$ and $\beta^2_{ij}<\bar b$. Let $\mathcal X$ be as above, and let $\mathcal B$ denote the manifold which is the largest open set of allowed $[B^1\;B^2]$ consistent with the bound $\bar b$ and the zero-nonzero pattern of entries.  Being open, $\mathcal X$ and $\mathcal B$ are necessarily both manifolds, as is $\mathcal X\times \mathcal B$. Let 
\[
f_{\beta}:\mathcal X\times \mathcal B\to \mathcal Y, (x,\beta)\mapsto y=\begin{bmatrix}[-D^1+(I_n-X^1-X^2)B^1]x^1\\
[-D^2+(I_n-X^1-X^2)B^2]x^2
\end{bmatrix}
\] 
be a smooth map of manifolds, with  $x = [(x^1)^\top, (x^2)^\top]^\top$, $\beta=\text{vec}\, [B^1, B^2]$, and $\mathcal{Y} \subseteq \Xi_{2n}$.

\begin{lemma} 
With notation as above, the matrix $\frac{\partial f_{\beta}(x,\beta)}{\partial (x,\beta)}$ has full row rank at any coexistence equilibrium.
\end{lemma}
\begin{proof}
We observe first
\begin{small}
\begin{equation}\label{eq:Jac1}
\frac{\partial f_{\beta}(x,\beta)}{\partial (x,\beta)}=\frac{\partial f_{\beta}(x^1,x^2,B^1,B^2)}{\partial (x^1,x^2,B^1,B^2)}=\left[\frac{\partial f_{\beta}(x^1,x^2,B^1,B^2)}{\partial (x^1,x^2)}\quad\frac{\partial f_{\beta}(x^1,x^2,B^1,B^2)}{\partial (B^1,B^2)}\right]
\end{equation}
\end{small}
As for the previous lemma, we focus on the second matrix. Consider an arbitrary $i = 1,2$ and observe that $B^ix^i=(I_n\otimes (x^i)^{\top}){\rm{vec}}[(B^i)^{\top}]$. As a preliminary calculation, assume that $B^i$ is positive. One can obtain the gradient of this expression with respect to the entries of $B^i$, these entries being ordered for convenience in the same way as occurs in ${\rm{vec}}[(B^i)^{\top}]$, i.e. as \[b^i_{11}.b^i_{12},\dots,b^i_{1n},b^i_{21},b^i_{22},\dots,b^i_{2n},\dots,b^i_{nn}.\]
and assuming temporarily all are nonzero, the gradient is
\begin{equation}\label{eq:fullcase}
    \frac{\partial (B^ix^i)}{\partial B^i}=\begin{bmatrix}
    [x_1^i\;x_2^i\dots x^i_n]&{\bf{0}}_n^{\top}&\dots &{\bf{0}}_n^{\top}\\
    {\bf{0}}_n^{\top}&[x_1^i\;x_2^i\dots x^i_n]&\dots&{\bf{0}}_n^{\top}\\
    \vdots\\
    {\bf{0}}_n^{\top}&{\bf{0}}_n^{\top}&\dots&[x_1^i\;x_2^i\dots x^i_n]
    \end{bmatrix}
\end{equation}
There is one column corresponding to each entry of $B^i$, with the first $n$ columns corresponding to the first row $[b^i_{11},b^i_{12},\dots,b^i_{1n}]$ of $B^i$, the second group of $n$ columns to the second row of $B^i$, and so on. In the event that some entries of $B^i$ are zero, the corresponding columns on the right of \eqref{eq:fullcase} drop out.

Recall that no row of an irreducible $B^i$ has all zero entries. Given the zero-nonzero entry pattern of our particular irreducible $B^i$, let $n^i_k \geq 1$ be the number of nonzero entries in the $k$-th row of $B^i$, for $k = 1, 2, \hdots, n$. Let $y^i_k \in \mathbb R^{n^i_k}$ be obtained from taking $x^i = [x^i_1, \hdots, x^i_n]^\top$ and deleting the $j$-th entry whenever $\beta^i_{jk} = 0$ (i.e. the $j$-th entry of the $k$-th row of $B^i$ that are zero). Then, there holds
\begin{equation}\label{eq:partialcase}
    \frac{\partial (B^ix^i)}{\partial B^i}=\begin{bmatrix}
    (y^i_1)^{\top}&{\bf{0}}_{n^i_{2}}^{\top}&\dots &{\bf{0}}_{n^i_{n}}^{\top}\\
    {\bf{0}}_{n^i_{1}}^{\top}&(y^i_2)^{\top}&\dots&{\bf{0}}_{n^i_{n}}^{\top}\\
    \vdots\\
    {\bf{0}}_{n^i_1}^{\top}&{\bf{0}}_{n^i_2}^{\top}&\dots&(y^i_n)^{\top}
    \end{bmatrix}
\end{equation}
The matrix clearly has full row rank $n$, as does the matrix in \eqref{eq:fullcase}.

Thus, we have
\begin{equation}\label{eq:full_derivative}
  \frac{\partial f_{\beta}(x^1,x^2,B^1,B^2)}{\partial (B^1,B^2)}=  \begin{bmatrix}
I-X^1-X^2&0\\
0&I-X^1-X^2\end{bmatrix}
\begin{bmatrix}
\frac{\partial (B^1x^1)}{\partial B^1}&0\\
0&\frac{\partial (B^2x^2)}{\partial B^2}\end{bmatrix}
\end{equation}
where $\frac{\partial (B^ix^i)}{\partial B^i}$ is given by \eqref{eq:fullcase} if $B^i$ is a positive matrix and by \eqref{eq:partialcase} if $B^i$ has one or more zero entries, for $i = 1,2$. Clearly, in either case the second matrix on the right of \eqref{eq:full_derivative} has rank $2n$.  

Of course,  in the interior of $\Xi_{2n}$, there holds $x^1\gg {\bf{0}}_n,x^2\gg{\bf{0}}_n$ and $x^1+x^2\ll {\bf{1}}_n$, and therefore the matrix $(I-X^1-X^2)$ is nonsingular. Hence  $\frac{\partial f_{\beta}(x^1,x^2,B^1,B^2)}{\partial (B^1,B^2)}$ has rank~$2n$. Since this is a submatrix of $\frac{\partial f_{\beta}(x,\beta)}{\partial (x,\beta)}$ by \eqref{eq:Jac1}, this means $\frac{\partial f_{\beta}(x,\beta)}{\partial (x,\beta)}$ has rank~$2n$. Therefore, for all $(x,\beta) \in \mathcal X\times \mathcal B$ for which $f_{\beta}(x,\beta)={\bf{0_{2n}}}$, the Jacobian of $f_{\beta}$ with respect to $x$ and $\beta$ is of full row rank.
\end{proof}

Let $\mathcal Z$ be the submanifold of $\mathcal Y$ consisting of the single point ${\bf{0_{2n}}}$. Evidently, it is in fact the image of those points in $\mathcal X \times \mathcal D$ defined by $f_{\delta}(x,\delta)={\bf{0_{2n}}}$ and also (separately considered) it is the image of those points in $\mathcal X \times \mathcal B$ defined by $f_{\beta}(x,\beta)={\bf{0_{2n}}}$.  Since the Jacobians of $f_{\delta}$ with respect to $(x,\delta)$ and $f_{\beta}$ with respect to $(x,\beta)$ have full rank, this means that the maps  $f_{\delta}$ and $f_{\beta}$ are both transversal to $\mathcal Z$. By the Parametric Transversality Theorem \cite[see p.145]{lee2013introduction}, \cite[see p.68]{guillemin2010differential}, it follows that for almost all choices $\bar\delta$ of $\delta$ and $\bar \beta$ of $\beta$, the mappings $f_{\bar\delta}:\mathcal X\to \mathcal Y, f_{\bar \delta}(x)=f_{\delta}(x,\bar \delta)$ and $f_{\bar\beta}:\mathcal X\to\mathcal Y,f_{\bar\beta}(x,\bar\beta) $ will be transversal to $\mathcal Z$, i.e. the Jacobians $\frac{\partial f_{\delta}(x,\bar \delta)}{\partial x}$ and $\frac{\partial f_{\beta}(x,\bar \beta)}{\partial x}$ will be of full rank $2n$  at the zeros of $f_{\bar\delta}$ and $f_{\bar\beta}$. Since these matrices are square and of size $2n\times 2n$, this says that  every zero of $f(x,\bar\delta)$ and of $f(x,\bar\beta)$ is nondegenerate, and as a consequence isolated. Since the zeros occur in a bounded set, this means they are finite in number. The choices of $\delta$ or of $\beta$ for which they are not finite in number, if indeed such choices exist, define a set of measure zero.

The above calculations, because of the definition of $\mathcal X$, only established that coexistence equilibria are isolated. We now examine equilibrium points on the boundary of $\Xi_{2n}$. We know of three, viz. the healthy equilibrium $({\bf{0}}_n, {\bf{0}}_n), (\bar x^1,{\bf{0}}_n)$ and $({\bf{0}}_n,\bar x^2)$, where $\bar x^i, i=1,2$ are equilibrium nonzero solutions of \eqref{eq:singlevirusunlabelled} with $(D,B)=(D^i,B^i), i=1,2$. There are in fact no others, \cite[see Lemma 3.1]{ye2021_bivirus}.

The nondegeneracy of the healthy equilibrium follows from verifying that the Jacobian of \eqref{eq:bivirus} evaluated at the healthy equilibrium is 
\begin{equation}
    J({\bf{0}}_n,{\bf{0}}_n)=\begin{bmatrix}
    -D^1+B^1&0\\0&-D^2+B^2
    \end{bmatrix}
\end{equation}
and Assumption \ref{ass:unstablehealthy}.  In particular, if the $B^i$ are fixed, it is clear that for generic $D^i$ subject to the constraint of Assumption \ref{ass:unstablehealthy}, the block diagonal matrices will be nonsingular, and similarly if the $D^i$ are fixed and $B^i$ generic. 

The Jacobian associated with $(\bar x^1,{\bf{0}}_n)$ was obtained in \eqref{eq:boundaryjacobian} and is nonsingular if and only if $-D^2+(I_n-\bar X^1)B^2$ is nonsingular. It is clear this property holds except possibly when $B^2$ is fixed for $D^2$ on a set of zero, and for $D^2$ fixed for $B^2$ on a set of measure zero. Of course a similar argument deals with the third remaining boundary equilibrium $({\bf{0}}_n,\bar x^2)$.

\section{Proof of Theorem \ref{app:pf_thm_hyperbolicity}}\label{sec:app b}

Since the two claims can be proved by effectively the same procedure, we will show just one of the two claims, viz. that for almost all $B^i$ and with fixed $D^1,D^2=I$, and provided that the number of equilibria is finite, any equilibrium is hyperbolic.  We will use a tool of algebraic geometry. Note that \cite{ye2021_bivirus} outlines the way algebraic geometry can be used to demonstrate that generically, there are a finite number of equilibria; the argument here will be similar. 

More specifically, suppose that the number of equilibria is finite. Let $J$ denote the Jacobian associated with \eqref{eq:bivirus}. With the right side of \eqref{eq:bivirus} denoted by $f(x)$, then $J(x)=\nabla f(x)$. If $\bar x$ is an equilibrium, it is hyperbolic if and only if $J(\bar x)$ has no eigenvalue with zero real part. Since for almost all parameter choices, there is no zero eigenvalue (because zeros of $f(x)$ are nondegenerate), we only need to rule out the possibility of there being pure imaginary eigenvalues of the Jacobian. However, instead of doing this, we shall rule out a situation of wider scope, viz. the possibility that there are eigenvalue pairs whose sum is zero; obviously, if (generically) there are no eigenvalue pairs summing to zero, this implies as a particular case that there are no pure imaginary eigenvalues, since they necessarily occur in complex conjugate pairs which sum to zero.
Taking all this together, we need to rule out existence for almost all $B^i$ of a solution 
$(\bar x,\bar s)$ with $\bar x \in \Xi_{2n}$ and $\bar s \in \mathbb C$  to the following equations:
\begin{align}\label{eq:zerosummming}
f(x)&=0\\\nonumber
|sI_{2n}-\nabla f(x)|&=0\\\nonumber
|sI_{2n}+\nabla f(x)|&=0
\end{align}

This set of $2n+2$ equations in the $2n+1$ scalar variables $(x,s)$ is a set of \textit{multivariate inhomogeneous polynomial} equations, with the coefficients of the monomials appearing in each equation being themselves polynomial (actually affine) in the entries of $D^i$ and $B^i$, for $i=1,2$. An associated set of homogeneous equations in $2n+2$ scalar variables, $(x,s,u)$ say, where $u$ is the homogenising variable can immediately be formed,  e.g. \cite[see pp. 84 ff.]{cox2006using}. If a solution of the latter equations exists with $u\neq 0$, a solution of the inhomogeneous equation exists. The only other solutions of the homogeneous equations {\color{cyan}with} $u=0$ are those associated with a nonzero solution of a {\color{cyan}modified} form of the inhomogeneous equations obtained by setting all terms other than those of highest degree  to zero (the so-called 'solutions at infinity').

Typically, $(2n+2)$ inhomogeneous equations in $(2n+1)$ unknowns might be expected to have no solution. As argued in detail in \cite{ye2021_bivirus}, there is however a condition for the existence of a solution to the associated homogeneous equations involving a \textit{resultant}; the resultant is a multivariate polynomial in the coefficients appearing in the equations, and there can only be a solution to the equations when the resultant evaluates as zero for the particular coefficient values. Since the coefficients are affine in the entries of the $B^i$, this means, crucially, that if for any fixed $D^i$, there is a particular choice of entries of the $B^i$ for which the resultant is nonzero then for that choice \textit{and consequently almost all other choices of the $B^i$ entries} the resultant will be nonzero and there can be no solution to the homogeneous equations or no solution $(\bar x,\bar s)$ to the inhomogeneous. equations. The fact that the particular choice of $B^i$ might have certain zero entries, or entries of certain signs, is irrelevant in drawing this conclusion. Complex solutions as well as real solutions are ruled out by a nonzero resultant, even though the coefficients in the equations are real. 

(The result generalizes the easily understood notion that if two real polynomials $g_1(x,b), g_2(x,b)$ have coefficients affine in some real scalar parameter $b$, and if there is no common zero (real or complex) of the polynomials for a particular $b$, there will be no common zero (real or complex) for almost all $b$.) 

We can now demonstrate that there exists a choice of the $B^i$ for which there is no solution of \eqref{eq:zerosummming}. Choose both $B^i$ to be diagonal, thus $B^i={\rm{diag}}[\beta^i_{11},\beta^i_{22},\dots, \beta^i_{nn}]$, with also $\beta^i_{jj}>1$ for all $j$. Also, require that $\beta^1_{jj}\neq \beta^2_{jj}$ for any $j$. It is then easily checked that the requirement $f(x)=0$ implies that for each $j$, there holds one of three possibilities, viz. $(\bar x^1_j,\bar x^2_j)=(0,0), (1-(\beta^1_{jj})^{-1}, 0), (0,1-(\beta^2_{jj})^{-1})
$. We must check that for each possibility, it is impossible, for almost all choices of $\beta^i_{jj}$, for $\nabla f(x)$ to have two eigenvalues summing to zero. To this end, we must evaluate $\nabla f(x)$  

It is straightforward to see that $\nabla f(x)$ with row and column reordering is a direct sum of $2\times 2$ matrices, of which the $j$-th is given by
\begin{equation}
G_j=\begin{bmatrix}
-1+(1-\bar x^1_j-\bar x^2_j)\beta^1_{jj}-\beta^1_{jj}\bar x^1_j&-\beta_{jj}^1\bar x^1_j\\
-\beta^2_{jj}\bar x^2_j&-1+(1-\bar x^1_j-\bar x^2_j)\beta^2_{jj}-\beta^2_{jj}\bar x^2_j
\end{bmatrix}
\end{equation}
For the three possible values of $\bar x^1,\bar x^2)$, we obtain
\begin{align}
    G_j&=\begin{bmatrix} 
    -1+\beta^1_{jj}&0\\0&-1+\beta^2_{jj}
    \end{bmatrix}
    \quad\mbox{for } (\bar x^1,\bar x^2)=(0,0)\\
    G_j&=\begin{bmatrix}
    -\beta^1_{jj}+1&-\beta^1_{jj}+1\\0&-1+\beta^2_{jj}/\beta^1_{jj}
    \end{bmatrix}
    \quad\mbox{for } (\bar x^1,\bar x^2)=(1-(\beta^1_{jj})^{-1},0)\\
    G_j&=\begin{bmatrix}
    -1+\beta^1_{jj}/\beta^2_{jj}&0\\-\beta^2_{jj}+1&-\beta^2_{jj}+1
    \end{bmatrix}
    \quad\mbox{for }(\bar x^1,\bar x^2)=(0,1-(\beta^2_{jj})^{-1})
\end{align}
It is clear that the eigenvalues of $\nabla f(x)$ will include two eigenvalues for each value of $j$, and these will either be $-1+\beta^1_{jj},-1+\beta^2_{jj}$, or $-\beta^1_{jj}+1,-1+\beta^2_{jj}/\beta^1_{jj}$ or $-1+\beta^1_{jj}/\beta^2_{jj}, -\beta^2_{jj}+1)$. For almost all choices of the $\beta^i_{jj}, i=1,2, j=1,2,\dots, n$, no two eigenvalues will sum to zero. es the hyperbolicity. To rule out the possibility that the resultant is zero, we must also exclude the possibility of a solution at infinity. It is straightforward to check that setting all but the highest degree terms in the equation set $f(x)=0$ to zero for the chosen $B^i$ leads to all $x^i_j$ being zero. When one sets the highest degree terms in the last two equations of \eqref{eq:zerosummming} to zero and imposes the requirement that $x=0$, there results $s=0$. Thus there is no nonzero solution to the modified inhomogeneous equations defining solutions at infinity of the homogeneous equations.  
The two calculations thereby establish the hyperbolicity.

\bibliographystyle{IEEEtran}
\bibliography{PoincareHopfbivirus}
\end{document}